\documentclass[12pt,a4paper]{amsart}

\usepackage{amsmath}
\usepackage{amssymb}
\usepackage{amsthm}
\usepackage{latexsym}
\usepackage{ytableau}
\usepackage[enableskew]{youngtab}
\usepackage{pgf,tikz}
\usetikzlibrary{arrows}
\usepackage{float}

\usepackage{amscd}
\usepackage{graphicx}
\definecolor{zzttqq}{rgb}{0.6,0.2,0}
\definecolor{qqqqff}{rgb}{0,0,1}

\theoremstyle{plain}
\newtheorem{theorem}{Theorem}[section]
\newtheorem{corollary}[theorem]{Corollary}
\newtheorem{lemma}[theorem]{Lemma}
\newtheorem{proposition}[theorem]{Proposition}
\newtheorem{conjecture}[theorem]{Conjecture}
\theoremstyle{definition}
\newtheorem{definition}{Definition}[section]

\theoremstyle{remark}
\newtheorem{remark}{Remark}[section]
\newtheorem{example}{Example}[section]

\numberwithin{equation}{section}

\numberwithin{table}{section}

\numberwithin{figure}{section}

\newcommand{\olga}{\color{blue}}

\begin{document}

\title[A classification  of  monotone   ribbons   with full Schur    support]{ A classification  of  monotone   ribbons   with full Schur    support \\ with application to the classification of full equivalence classes}

\author{Olga Azenhas}
\address{CMUC, Department of Mathematics, University of Coimbra, Apartado 3008,
3001--454 Coimbra, Portugal} \email{oazenhas@mat.uc.pt}

\author{Ricardo Mamede}
\address{CMUC, Department of Mathematics, University of Coimbra, Apartado 3008,
3001--454 Coimbra, Portugal} \email{mamede@mat.uc.pt}

\thanks{This work was partially supported by the Centre for Mathematics of the
University of Coimbra -- UID/MAT/00324/2013, funded by the Portuguese
Government through FCT/MCTES and co-funded by the European Regional
Development Fund through the Partnership Agreement PT2020, and by the FCT sabbatical
grant SFRH/BSAB/113584/2015. The  first author  wishes to acknowledge the hospitality of  University of Vienna where her sabbatical leaving in the academic year 2015/2016 took place and this work was partially developed.}



\maketitle
\begin{abstract}
We consider ribbon shapes, not necessarily connected, whose rows, with at least two boxes in each, are in  monotone length order.  These ribbons are uniquely defined by a pair of partitions: the row partition consisting of  the row lengths in decreasing order, and the overlapping partition whose entries count the total number of columns with two boxes in the successive  ribbon shapes obtained by sequentially subtracting   the longest row.
The support of such ribbon Schur functions,  considered as a subposet of
the dominance order lattice, has the row partition as bottom element, and,  as top element, the partition whose two parts consist of the total number of columns, and the total number of columns of length two respectively.  We give a complete system of
linear inequalities in terms of the partition pair defining the aforesaid ribbon shape
 under which  the ribbon Schur function attains  all the
Schur interval  when expanded in the basis of Schur functions. We then
conclude that the Gaetz-Hardt-Sridhar necessary condition for a connected ribbon to have  full equivalence class  is equivalent to the condition for a monotone connected ribbon to have  full Schur support. That is, the  set of partitions with full equivalence class is a subset of those   monotone connected ribbons with full Schur support.  M. Gaetz, W. Hardt and S. Sridhar  conjectured that the necessary condition is  also sufficient which translates now to  every monotone connected ribbon with full Schur support has full equivalence class. The main tool of our analysis is the structure of the companion tableau of a ribbon Littlewood-Richardson (LR) tableau  detected by the descent set defined by the composition whose parts are the ribbon row lengths.
\end{abstract}
\section{Introduction and statement of results}
Littlewood-Richardson (LR) coefficients, non negative integers, arise in a variety of areas of mathematics \cite{ful2}.
Determining its positivity without evaluating its actual value is of importance.  There  exists a variety of combinatorial models, collectively called
 Littlewood-Richardson rules (the original model conjectured in \cite{lr} and proved in \cite{sch,T2}) to compute LR coefficients, and to show their positivity it is enough to exhibit an object in a chosen combinatorial model. Linear inequalities on triples  of partitions guaranteeing their positivity  have arisen
from studying eigenvalues of a sum of Hermitian matrices \cite{horn,klyachko,knutao,ful2}. Given the skew partition $A:=\lambda/\mu$, with $\mu\subseteq \lambda$ partitions, it is known that it uniquely defines a subposet $[\rm r(A), \rm \rm c(A)']$ in the dominance order lattice of partitions of $|A|$, the number of boxes of $A$, where the  bottom element $\rm r(A)$ is the
  partition formed by the row lengths of $A$,
 and the top element $\rm c(A)'$ is  the conjugate of the partition $\rm c(A)$ formed by the column lengths of $A$. The meaning of this interval is that, given the partition $\nu$ of  $|A|$, the LR coefficient $c_A^\nu:=c_{\mu,\nu}^\lambda> 0$ only if  $\nu\in [\rm r(A),\rm c(A)']$ and, in particular, $c_A^{\rm r(A)}=c_A^{\rm c(A)'}=1$ (see,for instance, \cite{az,Mn} and references therein). Indeed it is not enough $\nu\in [\rm r(A),\rm c(A)']$ to guarantee that $c_A^\nu>0$ \cite{knutao, ful2}.

The LR coefficient $c_A^\nu$  is a structure coefficient. It arises, for example, as the multiplicity of the Specht module $S^\nu$  in the  decomposition of  the skew Specht module $S^A$ into irreducible representations  of the symmetric group $\sum_{|A|}$,
\begin{equation}\label{spech}S^{A}\cong\bigoplus_{\nu\in [\rm r(A),\rm c(A)']}(S^{\nu})^{\oplus{c_A^\nu}};\end{equation}
and, in the algebra of symmetric functions, as a coefficient of the Schur function $s_\nu$ in the expansion of the skew Schur function $s_A$
 in the basis of Schur functions $s_\nu$,
\begin{equation}\label{exp}s_{A}=\sum_{\nu\in [\rm r(A),\rm c(A)']}c_A^\nu s_{\nu}.\end{equation}
The expansion \eqref{exp} is also  the image of the character of $S^A$ under the Frobenius characteristic map.
Another way to look either at  expansions \eqref{spech} or \eqref{exp} is that given  $A=\lambda/\mu$, $\mu\subseteq \lambda$, they generate all possible positive LR coefficients $c_A^\nu:=c_{\mu,\nu}^\lambda$.
In view of these expansions, $[\rm r(A),\rm c(A)']$ is then the Schur interval of the skew shape $A$, and
the {\em Schur support} $[A]$ of the skew shape $A$ is   the set of partition shapes $\nu$ where either $S^\nu$ appears with positive multiplicity in \eqref{spech} or $s_{\nu}$ appears with nonzero coefficient in \eqref{exp},
\begin{equation}[A]:=\{\nu:c_{A}^{\nu}>0\}\subseteq [\rm r(A),\rm c(A)'].\label{interval}
\end{equation}
The skew shape $A$ is said to have {\em full Schur support} when in \eqref{interval} the support  coincides with  the Schur interval.

A very general problem in the calculus of shapes is the classification of skew shapes $A$ whose Schur support consists
of the whole interval $ [\rm r(A),\rm c(A)']$ in the dominance order lattice of partitions. (See also Question 5.1 in \cite[Section 5]{PvW}.) In other words, given the  partition $\nu$ of $|A|$, we ask  under which conditions  one has,
$c_{A}^{\nu}>0$ {if and only if } $ \nu\in [\rm r(A),\rm c(A)']$. In the special case  of requiring  all coefficients $c_{A}^{\nu}=1$,  the multiplicity free full interval, a classification was given in \cite{acm17}. We here give,  in Theorem \ref{corf}, a full Schur support classification for  monotone ribbon shapes, not necessarily connected,  with at least two boxes in each row, in terms of linear inequalities \eqref{fullineq}  satisfied by the partition pair $(\alpha,p)$ consisting of the row  and overlapping partitions  defining the monotone ribbon shape (see Proposition \ref{prop:biject}). The significance of this classification also amounts to the classification of connected ribbons with {\em full equivalence class} (\cite[Definition 7]{fullequiv}), that is, connected ribbons whose Schur support is invariant under  any order rearrangement of the rows. More precisely,  monotone connected ribbons with full equivalence class  only exist among those  with full Schur support. This is a recent  input on our study of  monotone ribbons having full Schur support and comes from the work by  Gaetz,  Hardt,   Sridhar and  Quoc Tran \cite{fullequiv,fullequiv2} where the support equality among connected ribbon Schur functions under any order rearrangement of the rows is addressed.
 The set of connected ribbons with full equivalence class  has partitions as ribbon representatives. Lemma \ref{lem:equiv} shows that the Gaetz-Hardt-Sridhar necessary condition \cite[Theorem II.1]{fullequiv} for connected ribbons to have full equivalence class is equivalent to our classification, in  Theorem \ref{corf}, of  monotone  connected  ribbons with full Schur support. Theorem \ref{fullequivschur} concludes that a  monotone connected ribbon with full equivalence class has   full Schur support. For monotone connected  ribbons with at most four rows, ribbons with full equivalence class coincide with ribbons with full Schur support.

 Earlier work on calculus of skew shapes are, for instance,  Schur support containments    by Pylyavskyy, McNamara and  van Willigenburg \cite{pylyavsky,PvW},  skew shapes with the same Schur support or   skew Schur function equalities by McNamara and   van Willigenburg \cite{Mn,PvW1}. In particular, ribbon Schur functions  were already considered by MacMahon \cite[ 199--202]{mac} and Foulkes \cite{foulkes} with
representation-theoretic significance by the last. Finally, it is worth noting that  Reiner, Shimozono \cite{rs} and  R. I. Liu \cite{liu} have considered  Specht modules and, therefore, Schur functions  for more general diagrams than skew shapes. However, apart  percentage-avoiding diagrams \cite{rs}, the combinatorial description of the coefficients for the Schur expansion is not known in general.

\subsection{ Overlapping partition of a monotone ribbon and descent set of a SYT} Arbitrary connected ribbons (diagrams corresponding to  skew shapes containing no $2\times 2$ rectangle)  are in bijection with compositions  assigning to the ribbon the row lengths. Thanks to the $\pi$-rotation symmetry of LR coefficients \cite{stanley,acm},  the Schur support classification  of LR monotone ribbons  may be reduced to ribbons with row lengths in monotone decreasing order.
Decreasing monotone ribbons with rows in length at least two, have at most columns of length two which occur exactly when two rows overlap: the overlapping partition $p$, read in reverse order, records sequentially, by accumulation, the number of columns of length two from the bottom  to the top rows of the ribbon (see  Section \ref{sec:ribbon}  and  Definition \ref{overlap}).
Proposition \ref{prop:biject} shows that monotone ribbons, not necessarily connected, with at least  two boxes in each row in monotone length order, are in bijection, up to an antipodal rotation,  with   partition pairs $(\alpha,p)$ where the $\ell(\alpha)$ parts of the row lengths partition $\alpha=(\alpha_1,\dots,\alpha_{\ell(\alpha)})$ are in length at least two, and the $\ell(p)$ parts of the overlapping partition $p=(p_1,\dots,p_{\ell(\alpha)-1},0)$ are assigned  by  a multiset of $\{\ell(\alpha)-k,\dots,2,1\}$ of cardinality $\ell(\alpha)-k\le \ell(p)\le \ell(\alpha)-1$ with $k\in\{1,\dots,\ell(\alpha)\}$.  We often denote these ribbons by $R_\alpha^p$, or just  say  the partition $\alpha$ with overlapping partition $p$    to mean that $p$ is the overlapping partition of the ribbon  $R^p_\alpha$.  The   Schur interval of our ribbon  $R_\alpha^p$, with $k=\ell(\alpha)-p_1$ connected components, is
\begin{equation}\label{partsupp}[R_\alpha^p]\subseteq [\alpha,(|\alpha|-\ell(\alpha)+k,\ell(\alpha)-k)].\end{equation}

\begin{example} \label{example:first}The partition pair $(\alpha=(3,3,2,2,2), p=(2,2,1,1,0))$ where $p_1= \ell(\alpha)-3=2$ and $\ell(p)=\ell(\alpha)-1=4$, defines the monotone ribbon $R_{(3,3,2,2,2)}^p$, below, with 3 connected components, and Schur interval  $ [(3,3,2,2,2),(10,2)]$,
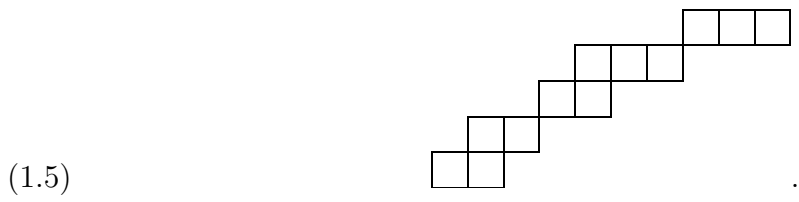
\begin{equation}\label{ex:first}\young(:::::::\hfill\hfill\hfill,::::\hfill\hfill\hfill,:::\hfil\hfill,:\hfill\hfill,\hfill\hfill).\end{equation}
\end{example}

Our classification is based on the fact that given  a monotone ribbon with row lengths at least two, defined by the partition pair $(\alpha,p)$, the existence of a companion tableau \cite[Appendix]{leclen,Nak05} for an  LR  filling of $R^p_\alpha$ with content $\nu$,   is equivalent to show that the triple of partitions $\alpha$,  $p$ and $\nu$  satisfy a certain system of linear inequalities \eqref{ineq} in Theorem \ref{ThmNec}.
 The companion tableau of a LR  connected ribbon $R_\alpha$ is detected by the descent set  $S(\alpha)=\{\alpha_1,\alpha_1+\alpha_2,\dots,\alpha_1+\cdots+\alpha_{\ell(\alpha)-1}\}$ of its standardization (see sections \ref{sec:standard} and \ref{subsec:descent}).   The following alternative description of the LR coefficients in the expansion  \eqref{exp} is known \cite{foulkes,gessel84,gessel93}, counting exactly standardized companion tableaux of connected LR ribbons.
\begin{theorem} \label{th:gessel} \cite{foulkes,gessel84,gessel93}. Let
 $\alpha$ be any composition of $N$ and $R_\alpha$ the corresponding connected ribbon shape. Then
$$s_{R_\alpha}=\sum_{\nu}d_{\nu,\alpha}s_{\nu},$$
where $\nu$ runs on the set of partitions of $N$, and $d_{\nu,\alpha}$ is the number of standard Young tableaux (SYT) of shape $\nu$ and descent set $S(\alpha)=\{\alpha_1,\alpha_1+\alpha_2,\dots,\alpha_1+\cdots+\alpha_{\ell(\alpha)-1}\}.$
\end{theorem}
This means that given the connected ribbon $R_\alpha$, the  LR ribbon coefficient $c^\nu_{R_\alpha}=d_{\nu,\alpha}$ is positive if and only if  there exists a semistandard Young tableau (SSYT) tableau of shape $\nu$ and content $\alpha$ whose standardization  has descent set $\mathcal{S}(\alpha)=\{\alpha_1,\alpha_1+\alpha_2,\dots,$ $\alpha_1+\cdots+\alpha_{\ell(\alpha)-1}\}$. For ordered compositions with parts of length at least two, we show, in Theorem \ref{ThmNec}, that the existence of such standard Young tableau guaranteeing the positivity of $c^\nu_{R_\alpha}$ is equivalent  to require that the triple of partitions $\alpha$, $\nu$ and $p=(\ell(\alpha)-1,\dots,2,1,0)$ satisfy a certain  system of linear  inequalities \eqref{ineq}. More generally, we prove that the characterization is valid for  monotone ribbons with $k$ components  by replacing the stair partition  $p$ of $\ell(\alpha)-1$ with a multiset of $\{\ell(\alpha)-k,\dots,1\}$ of cardinality $\ell(\alpha)-k\le \ell(p)\le \ell(\alpha)-1$, where $k\in\{1,\dots,\ell(\alpha)\}$.
Our method then consists of explicitly identifying in a SSYT of shape $\nu$ and content the partition $\alpha$, the obstructions for being a companion tableau for a monotone LR ribbon, with the goal to  remove them through a {\em rotation procedure} (see Subsection \ref{sec:rotation}). This removal is possible  whenever  linear inequalities \eqref{ineq} are satisfied by the triple of partitions $(\alpha,p,\nu)$. More precisely,
the {\em effective} obstructions, detected by the overlapping partition $p$, correspond to some  elements in  $\mathcal{S}(\alpha)$ which are not in the descent set of the standardized tableau. Thus to exhibit the positivity of a such LR ribbon coefficient one just needs to exhibit a companion tableau for the ribbon LR filling. To minimize the number of obstructions that we have to deal with we work out on  a SSYT with canonical filling (see Section \ref{sec:canonicalfilling}).

\subsection{ Monotone ribbons: witness vectors and their slacks}
Put $x_+:=max\{0,x\}$ where $x$ is a real number. To a monotone ribbon $R_\alpha^p$, we associate a sequence  $\{\tilde g^i\}_{i=1}^{\ell(p)-1}$ of $ \ell(p)-1$  {\em witness vectors}, and  to each witness $\tilde g^i$  we assign  the {\em slack} $p_{i+1}-1$,  
 for $ i\in\{1,\dots, \ell(p)-1\}$.
\begin{definition}\label{witness} Let $\alpha$ be a partition with parts at least two and with overlapping partition $p$. For each $ i\in\{1,$$\dots, $ $\ell(p)-1\}$,  put $\displaystyle \varrho_i-1:=\sum_{q=i+1}^{\ell(\alpha)}\alpha_q-p_{i+1}>0$ the {\em rest of order} $i$  of $R^p_\alpha$, that is, the total number of columns in the last $\ell(\alpha)-i$ rows of $R^p_\alpha$. Define the {\em $i$-witness vector } of $R_\alpha^p$  to be the nonnegative vector $\tilde g^i=(\tilde g_1^i,\dots,\tilde g_i^i)$ where
$\displaystyle \tilde g_j^i:=\left[ \varrho_i-\alpha_j\right]_+,\;j=1,\dots, i.$ The {\em slack} of the $i$-witness vector  is $p_{i+1}-1$, for $ i\in\{1,\dots, \ell(p)-1\}$.
If $\ell(p)=0,1$, $R^p_\alpha$ has no witness vectors.
\end{definition}

The  size $|\tilde g^i|:=\sum_{j=1}^i\left[ \varrho_i-\alpha_j\right]_+$ of the $i$-witness vector $\tilde g^i$  is said to fit its slack, if $|\tilde g^i|\le p_{i+1}-1$, otherwise is said to be oversized.
\begin{remark} For $ i\in\{1,$$\dots, $ $\ell(p)-1\}$, $\varrho_i$ exceeds the total number of columns in the last $\ell(\alpha)-i$ rows of $R^p_\alpha$. In any LR filling of $R_\alpha^p$ the $i+1$'s are filled in the last $\ell(\alpha)-i$ rows, and thereby its number is $<\varrho_{i}$. For $ i\in\{1,$$\dots, $ $\ell(p)-1\}$, $\tilde g^i=0$ if and only if $\alpha_i\ge\varrho_i$.
\end{remark}

\subsection{Statement of main results} Our  key result is Theorem \ref{ThmNec} which determines $c_{R_{\alpha}^p}^{\nu}>0$ without determining its actual value.
 It gives a set of linear inequalities on the partition triple  $(\alpha, p,\nu)$ as necessary and sufficient conditions for the positivity of $c_{R^p_{\alpha}}^\nu$. The inequalities are explained by the combinatorial interpretation of $\alpha\preceq\nu$ in the dominance order on partitions (see Remark \ref{re:ineqdom}),  and the obstruction of the overlapping partition $p$ to the partitions dominating $\alpha$.  When $p=0$, we have no such obstruction,  $c_{R_{\alpha}}^{\nu}$ is a Kostka number, and $\alpha\preceq\nu$ characterizes completely the aforesaid positivity.
\begin{theorem}\label{ThmNec} Let  $\alpha$ be a partition with parts at least two and overlapping partition $p=(p_1,\dots,p_{\ell(\alpha)-1},0)$, and $\nu$ a partition  of $|\alpha|$. Then
   \begin{equation}\label{ineq}c_{R_{\alpha}^p}^{\nu}>0 \Leftrightarrow \begin{cases}
\nu\in[\alpha,(|\alpha|-p_1,p_1)],\\
 \nu_i\leq {\displaystyle \sum_{q= i}^{\ell(\alpha)}\alpha_q-p_i},\,\text{ for }1\leq i\leq \ell(p).
\end{cases}
\end{equation}
In particular, when $\ell(p)=\ell(\alpha)-1$, that is, $p_i=\ell(\alpha)-i$,  $1\leq i\leq \ell(\alpha)$, there exists a SYT of shape $\nu$ with descent set $\mathcal{S}(\alpha)$ if and only if the right hand side of \eqref{ineq} is satisfied.
\end{theorem}
The necessary and sufficient condition \eqref{ineq} is easily read: $\nu\in[\alpha,(|\alpha|-p_1,p_1)]$ is in the support of $R^p_\alpha$ if and and only if the $\nu_i<\varrho_{i-1}$, with $\varrho_0:=|\alpha|-p_1+1$,
for $i=1,\dots,\ell(p)$.
With this on hand we give  a criterion to decide  when $R_\alpha^p$ has full Schur support, that is,  when one has $c_{R_{\alpha}^p}^{\nu}>0$ if and only if $\nu\in [\alpha,(|\alpha|-p_1,p_1)].$
 The test  assigns  to each $ i\in\{1,\dots, \ell(p)-1\}$ the $i$-witness vector of $R_\alpha^p$ and compares its size  with the  slack $p_{i+1}-1\ge 0$.
 The existence of a single witness fitting its slack prevents the full Schur support because it can be used to construct a partition in the Schur interval but not in the support. This is the case of a witness of size zero, that is, when the partition $\alpha$ has $\alpha_i\ge \varrho_i$ for some $1\le i\le \ell(p)-1$.
\begin{theorem}\label{Tchar} Let  $\alpha$ be a partition with parts $\ge 2$, and  overlapping partition $p=(p_1,\dots,p_{\ell(\alpha)-1},0)$. Then
$[R^p_\alpha]\subsetneqq [\alpha,(|\alpha|-p_1,p_1)]$ if and only if $\ell(p)\ge 2$ and, for some $1\le i\le \ell(p)-1$,  the  size of the $i$-witness vector $\tilde g^i$  fits its slack,   that is,
\begin{equation}\label{lineq}\sum_{j=1}^i\left[ \varrho_i-\alpha_j\right]_+\le p_{i+1}-1.\end{equation}
In this case,
$
\alpha_j +\tilde g_j^i\ge\varrho_i\ge p_{i+1}-1-|\tilde g^i|$, $j=1,\dots,i$, whose decreasing rearrangement is  the partition
$\displaystyle(\alpha_1 +\tilde g_1^i,\dots,
\alpha_i +\tilde g_i^i,\varrho_i,$ $p_{i+1}-1-|\tilde g^i| )^+$   of $|\alpha|$ in the Schur interval  of $R_\alpha^p$ but  not in the support of $R_\alpha^p$.
\end{theorem}
The equivalent statement for full Schur support is
 \begin{theorem} \label{corf}
Let $\alpha$  be a partition with parts $\ge 2$, and  overlapping partition $p=(p_1,\dots,p_{\ell(\alpha)-1},0)$. Then
$[R_\alpha^p]=[\alpha,(|\alpha|-p_1,p_1)]$ if and only if
 either $\ell(p)<2$ or $\ell(p)\ge 2$ and, in this case, for every $1\le i\le \ell(p)-1$, the  $i$-witness vector of $R^p_\alpha$  is  oversized with respect to its slack, that is,
\begin{equation}\label{fullineq}\sum_{j=1}^i\left[ \varrho_i-\alpha_j\right]_+\ge p_{i+1},\;\;1\le i\le \ell(p)-1.\end{equation}
\end{theorem}

\begin{remark} \label{aftertheor}
$R_\alpha^p$ has full support only if $$\alpha_i< \varrho_i\Leftrightarrow \alpha_i\le \sum_{q=i+1}^{\ell(\alpha)}\alpha_q-p_{i+1},\;\;\text{  $1\le i\le \ell(p)-1$}.$$
\end{remark}
The following is a generalization of \cite[Theorem 3.6]{fullequiv2} to monotone disconnected ribbons with $\ell(p)\le 3$ which contain the monotone connected ribbons of length $\le 4$.
\begin{corollary}\label{cor:full}
In particular,

$(a)$ when  $p=(2,1,0^{\ell(\alpha)-2})$,  $[R_\alpha^p]=[\alpha,(|\alpha|-2,2)]$ if and only if
 \begin{equation}\label{fullineq2}\alpha_1<\varrho_1\Leftrightarrow \alpha_1<\sum_{q=2}^{\ell(\alpha)}\alpha_q.
\end{equation}
$(b)$ when  $p=(3,2,1,0^{\ell(\alpha)-3})$,  $[R_\alpha^p]=[\alpha,(|\alpha|-3,3)]$ if and only if
\begin{equation}\label{fullineq3}\alpha_1<\sum_{q=2}^{\ell(\alpha)}\alpha_q-2\;\;\text{and}\;\;
\alpha_2<\sum_{q=3}^{\ell(\alpha)}\alpha_q.\end{equation}
\end{corollary}

In \cite[Theorem II.1]{fullequiv}, that we reproduce below as Theorem \ref{th:1.2} for the reader convenience, a necessary condition  is given for a connected ribbon with parts at least two, to have  full equivalence class \cite[Definition 7]{fullequiv}.
This necessary condition combined with Theorem \ref{corf} shows that a monotone connected ribbon with parts $\ge 2$ has full equivalence class only if it has full Schur support. That is, full equivalence classes only exist among monotone connected ribbons with full Schur support.
\begin{theorem}\label{th:1.2}\cite[Theorem II.1]{fullequiv} Let $\alpha$ be a partition with parts $\ge 2$ and $R_\alpha $ a connected ribbon. If $\alpha$ has full equivalence class then \begin{equation}\label{N} N_j:=max\{k:\sum_{\begin{smallmatrix}1\le i\le j\\\alpha_i<k\end{smallmatrix}}(k-\alpha_i)\le \ell(\alpha)-j-2\}<\varrho_j, \;1\le j\le \ell(\alpha)-2.\end{equation}
\end{theorem}
For monotone connected ribbons, inequality \eqref{N} is equivalent to  inequality \eqref{fullineq} in Theorem \ref{corf} characterizing full Schur support.
\begin{lemma}\label{lem:equiv} For all $ j\in\{1,\dots, \ell(\alpha)-2\}$,\begin{equation} N_j:=max\{k:\sum_{\begin{smallmatrix}1\le i\le j\\\alpha_i<k\end{smallmatrix}}(k-\alpha_i)\le \ell(\alpha)-j-2\}<\varrho_j\Leftrightarrow \sum_{\begin{smallmatrix}1\le i\le j\\\alpha_i<\varrho_j\end{smallmatrix}} (\varrho_j-\alpha_i)\ge \ell(\alpha)-j-1.\end{equation}
\end{lemma}
In addition, combining  Theorem \ref{corf} with  \cite[Theorem 3.6]{fullequiv2}, one has
\begin{theorem} \label{fullequivschur}Let $\alpha$ be a partition with parts $\ge 2$ and $R_\alpha $ a connected ribbon. If $\alpha$ has full equivalence class then $R_\alpha$ has full support.  When $\ell (\alpha)\le 4$, $\alpha$ has full equivalence class if and only if  $R_\alpha$ has full support.
\end{theorem}
Proofs of main results  will be delayed until sections \ref{sec:positivity}, \ref{sec:full} and \ref{sec:fullequiv}.

\subsection{Organization of the paper} This paper is organized in seven sections with the following contents.
The next section, divided in seven subsections, contains the basic  terminology, definitions and results that we shall be using throughout the paper. We highlight the concepts of  descent set of a semistandard Young tableau {\em versus} SYT and Proposition \ref{prop:desc} in Subsection \ref{subsec:descent}, the combinatorial interpretation of dominance order on partitions, in Subsection \ref{subsec:domin}, enlightening inequalities \eqref{ineq}, and companion tableau of an LR tableau, in Subsection \ref{subsec:comp}, our key tool in the  proof of the existence of a monotone ribbon LR filling with given shape and content or the positivity of a ribbon LR coefficient.

 Section \ref{sec:ribbon} is divided in four subsections. Subsection \ref{subsec:overlap} defines (Definition \ref{overlap}) and discusses overlapping partition of a ribbon, with row lengths at least two, that we shall use in  the (connected or not) monotonic case, and, in the last section, in the  connected case with row lengths in any order. It is shown that  monotone ribbons not necessarily connected   are uniquely defined by  the row lengths partition and the overlapping partition. It is recalled in Subsection \ref{subsec:ribcomp} that the descent set  of a standard Young tableau  detects the companion tableau of a LR connected ribbon.
 The enumerative characterization of LR connected ribbon coefficients $c_{R_\alpha}^\nu$ in Theorem \ref{th:gessel} is generalized  to disconnected ribbons.

 Given $T$ a SSYT of shape $\nu$ and weight $\alpha$
the descent set of the standardization of $T$ is a subset of $\mathcal{S}(\alpha)$. As our study reduces to ribbons $R_\alpha$ with $\alpha$ a partition, the serious rejection for $T$ to be a companion tableau for a LR ribbon of shape $R_\alpha$ occurs when it leads to a filling of $R_\alpha$ with the same  letter in a  column of length two.  In Subsection \ref{sec:critic}, we   translate the numbers  in $\mathcal{S}(\alpha)$ and not in the descent set of the standardized $T$, giving rise to the aforesaid violation, to  the {\em critical numbers set of $T$},  a subset of $\{2,\dots,\ell(\alpha)\}$.
In addition, as our monotone ribbons may be disconnected,  the overlapping partition is used to detect the effectiveness of the critical numbers of a companion tableau  of a LR ribbon of shape $R_\alpha^p$, as explained in Subsection \ref{sec:efectcritic}.

Section \ref{sec:positivity} gives the proof of Theorem \ref{ThmNec} which determines by means of a set of linear inequalities on the partition triple $(\alpha, p, \nu)$,  the positivity $c_{R_{\alpha}^p}^{\nu}>0$ without determining its actual value.
Assuming the linear inequalities on the  right hand side of  \eqref{ineq}, the goal is  to exhibit  a companion tableau for a LR filling of the shape $R^p_\alpha$.
   The semistandard tableau  of shape $\nu$ and weight $\alpha$ with {\em canonical filling} (Subsection \ref{sec:canonicalfilling})  is picked,  and then if necessary one modifies its filling according to a certain {\em rotation} procedure to avoid $p$-effective critical numbers so that the new tableau  is a companion tableau of an LR filling with weight $\nu$ of the shape $R^p_\alpha$.
The linear inequalities on the right hand side of \eqref{ineq} guarantee that our rotation procedure is successful.
Section \ref{sec:full} gives the proof of Theorem \ref{Tchar} and Theorem \ref{corf}, logically equivalent, which classify the monotone ribbons with full Schur support, and Corollary \ref{cor:full} which gives a simple version of those inequalities in the case where the overlapping partition has at most length four. Illustrative examples are also provided.

In section \ref{sec:fullequiv}, the bridge between the classification of monotone connected ribbons with full Schur support and those with full equivalence class \cite{fullequiv} is established. More precisely, Lemma \ref{lem:equiv} shows  that for monotone connected ribbons, the inequality \eqref{N}, in Theorem \ref{th:1.2}, \cite[Theorem II.1]{fullequiv}, giving a necessary condition for full equivalence class, is equivalent to the inequality \eqref{fullineq}, in Theorem \ref{corf}, characterizing the full Schur support. The bridge allows to prove  Theorem \ref{fullequivschur} which states that every partition with full equivalence class has full Schur support. Instances on the coincidence of these two classifications are provided.
More importantly, Corollary \ref{cor:length3} shows, as observed in Remark \ref{compfull},  that a non monotone connected ribbon of length three may have full Schur support while its monotone rearrangement does not have.

Section \ref{conj} generalizes,  in Theorem \ref{ineqcomposition}, the necessary condition, in Theorem \ref{ThmNec}, for the LR coefficient $c_{R_{\alpha}}^{\nu}$ positivity, with $\alpha$  a partition,  to  connected ribbons $R_{\beta}$ with $\beta$  a composition. Remark \ref{re:conj} shows that if these inequalities on the triple $(\beta, p, \nu)$ with $\beta$ a composition and $p$ the overlapping partition of $R_\beta$, are also sufficient, then the classification on partitions having full equivalence class and full Schur support is the same, and, henceforth, the Gaez-Hardt-Shridar conjecture \cite[Conjecture II.4]{fullequiv} claiming that the necessary condition \eqref{N} for a partition to have full equivalence class is also sufficient, is true.

{\bf  Acknowledgements.} We are thankful to   the  organizers of workshop Positivity in Algebraic Combinatorics, BIRS, Banff, Alberta, August 14-16, 2015, for the opportunity to present our work on full Schur supports,  to Jo\~ao Gouveia for useful discussions and suggesting the phrasing of witness vector with its slack which allowed  economy and clarification in our redaction, and  to M. Gaez, W. Hardt, S. Sridhar and P. Pylyavskyy for letting us know the paper \cite{fullequiv} on full equivalence classes.

\section{Preliminaries}
\label{sec:pre}
 \subsection{Partitions, compositions  and tableaux}
 \label{sec:standard}A partition $\lambda$ is an ordered  list of positive integers $\lambda_1\ge \lambda_2\ge\cdots\ge \lambda_{\ell(\lambda)}>0$ where $\lambda_i$ are the {\em parts} and  $\ell(\lambda)$ the {\em length} of $\lambda$.
We say that $|\lambda|:=\sum_{i=1}^{\ell(\lambda)}\lambda_i$ is the {\em size} of $\lambda$ and that $\lambda$ is a partition of $|\lambda|$. It is convenient to set $\lambda_k=0$ for $k> \ell(\lambda)$.
The {\em Young diagram} of the partition $\lambda=(\lambda_1,\lambda_2,\ldots,\lambda_{\ell(\lambda)})$, or Young diagram of shape $\lambda$,   is the collection of $|\lambda|$ boxes arranged in $\ell(\lambda)$ left-aligned rows, in the lower right quadrant of the plane, where the $i$th row has $\lambda_i$ boxes, for $1\leq i\leq \ell(\lambda)$. We shall identify a partition with its Young diagram. Given the partition $\lambda$, the conjugate or transpose partition $\lambda'$  is the  partition obtained by transposing the Young diagram of $\lambda$.
A filling $T$ of a Young diagram of shape $\lambda$ with positive integers is called {\em semistandard} if the integers increase weakly across rows (row semistandard condition) and strictly down columns (column standard condition). Such a filled-in Young diagram of shape $\lambda$ is called a {\em semistandard Young tableau} (SSYT) $T$  of shape $\lambda$. The {\em weight} or {\em content} of a SSTY is the sequence  $\alpha=(\alpha_1,\alpha_2,\ldots)$, where $\alpha_i$ is the number of integers $i$ in the filling of the tableau.

A {\em composition} $\alpha$ with $\ell(\alpha)$ parts is a sequence of $\ell(\alpha)$ positive  integers.
The partition $\alpha^+$  is the monotone nonincreasing rearranging of $\alpha$.
The size of $\alpha$ is defined to be $|\alpha|:=|\alpha^+|$, in which case we say $\alpha$ is a composition of $|\alpha|$.  The length of $\alpha$ is $\ell(\alpha)=\ell(\alpha^+)$.  If $\beta=(\beta_1,\dots,\beta_{\ell{(\beta)}} )$ is another composition,  we define  the {\em concatenation of $\alpha$ and $\beta$} to be the composition $\alpha.\beta=(\alpha_1,\ldots,\alpha_{\ell(\alpha)},\beta_1,\dots,\beta_{\ell(\beta)})$ of length $\ell(\alpha)+\ell(\beta)$.

 We denote by $Tab(\lambda,\alpha)$ the set of all SSYTs of shape $\lambda$ and content the {\em composition} $\alpha$.
For $\lambda$ a partition and $\alpha$ a composition of $|\lambda|$, the {\em Kostka number} $K_{\lambda,\alpha}$ is defined to be
$K_{\lambda,\alpha}:=\#Tab(\lambda,\alpha)$.

A {\em skew shape} or ({\em skew Young diagram})  $\lambda/\mu$ is  obtained by removing the Young diagram  $\mu$ from the top-left corner of the Young diagram $\lambda$, when  $\mu$ is contained in $\lambda$  as Young diagrams, or equivalently, when $\mu_i\le \lambda_i$, for all $i\ge 1$. In particular,  when $\mu$ is the empty partition $0$, we have $\lambda/0=\lambda$. The size of $\lambda/\mu$ is $|\lambda/\mu|:=|\lambda|-|\mu|$.
 An {\em horizontal strip} is a skew diagram which has at most one box in each column.
  The {\em basic form} of a skew shape is the skew diagram obtained  by deleting any empty row and any empty column.
 The skew shape $\lambda/\mu$ in the basic form defines the composition $\lambda-\mu$ that we simply write $\lambda/\mu$  if there is no danger of confusion.
A skew shape is said to be {\em connected} if there exists a path between any two boxes of the diagram using only north, east, south and west steps such that the path is contained in the diagram.
A SSYT of skew shape $\lambda/\mu$ and weight $\nu$ is a  semistandard filling of the the skew-shape $\lambda/\mu$ of weight $\nu$.
\subsection{Descent set of a standard tableau}
\label{subsec:descent}
If a SSYT $T$ of size $n$ ($n$ boxes) has entries in $[n]:=\{1,2,\dots,n\}$, each necessarily
appearing exactly once, then $T$ is said to be a {\em standard Young tableau} (SYT).

A SSYT $T$ in $Tab(\lambda,\alpha)$  may also be regarded as  a sequence  $0=\lambda^0\subseteq \lambda^1\subseteq\cdots \subseteq\lambda^{\ell(\alpha)}=\lambda$ of partitions such that  each skew shape $\lambda^i/\lambda^{i-1}$  is an horizontal strip of size $\alpha_i$. Simply insert an $i$ in each box of the strip $\lambda^i/\lambda^{i-1}$ \cite{stanley}.
The {\em standard order} on a semistandard Young tableau
 is the numerical ordering of the labels with priority, in the case of equality,
given by the rule southwest=smaller, northeast=larger. The standardization $\widehat T$ of a semistandard tableau $T\in Tab(\lambda,\alpha)$  is the enumeration of the
 labeled boxes according to the standard order of $T$, that is, the enumeration of the boxes across the sequence  $0=\lambda^0\subseteq \lambda^1\subseteq\cdots \subseteq\lambda^{\ell(\alpha)}$ where each horizontal strip $\lambda^i/\lambda^{i-1}$ of size $\alpha_i$ is  read SW-NE.
For instance,
the following are SSYT's with shape $\lambda=(4,3,2)$ and content $\alpha=(2,4,2,1)$, and their standardizations, respectively:
\begin{equation}\label{tab}T=\young(1122,223,34)\;  Q=\young(1122,224,33)\in Tab(\lambda,\alpha),\;  \quad \widehat T=\young(1256,348,79)\; \widehat Q=\young(1256,349,78).\end{equation}

The
{\em descent set} $\mathcal{D}(U)$ of a SYT $U$ of shape $\lambda$ is defined  to be the subset of $[|\lambda|-1]$ formed by those entries $i$ of $U$ for which $i + 1$ appears in
a strict lower row  of $U$ than $i$. There is a one-to-one natural correspondence between subsets of $[|\lambda|-1]$ and compositions of $|\lambda|$ \cite{stanley,foulkes}.
The composition $\alpha=(\alpha_1,\dots,\alpha_{\ell(\alpha)})$
  gives rise to the  subset $\mathcal{S}(\alpha):=\{\alpha_1, \alpha_1+\alpha_2,\ldots,\alpha_1+\alpha_2+\cdots+\alpha_{\ell(\alpha)-1}\}$, with cardinality $\ell(\alpha)-1$, of $[|\alpha|-1]$, and {\em vice-versa}. Hence a SYT of shape $\lambda$
has descent set $\mathcal{S}(\alpha)$ for some composition $\alpha$ of $|\lambda|$.
  In \eqref{tab}, for example, \begin{equation}\mathcal{D}(\widehat T)=\{2,6,8\}=\mathcal{S}(\alpha)\quad\text{ and}\quad
\mathcal{D}(\widehat Q)=\{2,6\}\subsetneqq\mathcal{S}(\alpha).\end{equation} However, $\widehat Q$ is also the standardization of $V=\young(1122,223,33)\in Tab(\lambda,\beta=(2,4,3))$ with $\beta=(\alpha_1,\alpha_2,\alpha_3+\alpha_4)$. In particular,  $Tab((4,0), (2,2))$ has a sole element whose standardization has   descent set the empty set, and  $Tab((4,4), (4,2,2))$ has a sole element
whose standardization has  descent set $\{4\}\varsubsetneqq \mathcal{S}(4,2,2)$.

Given  $T\in Tab(\lambda,\alpha)$, the {\em descent set $\mathcal{D}(T)$} of the SSYT $T$ is   the subset $\mathcal{S}$ of $\{1,\dots,\ell(\alpha)-1\}$ that consists of $s\in\{1,\dots,\ell(\alpha)-1\}$ for which there exists  a pair of entries  $s$ and $s+1$ in $T$ such that  $s+1$ appears in a strict lower row of $T$ than $s$. When $T$ is a SYT, that is, $T\in Tab(\lambda,(1^{|\lambda|}))$, we recover the notion of descent set in a SYT, where $\mathcal{D}(T)$ is a subset $\mathcal{S}$ of $[|\lambda|-1]$. We show next that
 a SYT of shape $\lambda$ has descent set $\mathcal{S}(\alpha)$ if and only if it is the standardization of some SSYT in $Tab(\lambda,\alpha)$ with descent set $\mathcal{ S}=\{1,\dots,\ell(\alpha)-1\}$.
A SYT of shape $\lambda$ has  descent set  $  \mathcal{S}(\beta)\subseteq \mathcal{S}(\alpha)$ if and only if it is the standardization of some SSYT in $Tab(\lambda,\alpha)$ with descent set  $\mathcal{ S}\subseteq\{1,\dots,\ell(\alpha)-1\}$
 and $\mathcal{S}(\beta)=\{\sum_{j=1}^s\alpha_j: s\in \mathcal{S}\}\subseteq \mathcal{S}(\alpha)$.

\begin{proposition} \label{prop:desc} Given  a partition $\lambda$ and a composition  $\alpha$ of $|\lambda|$, there exists a bijection between $Tab(\lambda,\alpha)$ and the set of all
 SYT's of shape $\lambda$ with descent set a subset of $ \mathcal{S}(\alpha)$,  defined by the map $T\mapsto \widehat T$.
 Moreover, if $T\in Tab(\lambda,\alpha)$ and $\mathcal{S}=\{s_1<\cdots<s_{|\mathcal{S}|}\}$
 then  $\mathcal{D}(\widehat T)=\{\sum_{j=1}^s\alpha_j: s\in \mathcal{S}\}=\mathcal{S(\beta)}\subseteq \mathcal{S}(\alpha)$
 with $\beta=(\beta_1,\dots,\beta_{|\mathcal{S}|}, |\alpha|-\beta_{|\mathcal{S}|})$ such that $\beta_i-\beta_{i-1}=\sum_{j=1}^{s_i}\alpha_j-\sum_{j=1}^{s_{i-1}}\alpha_j$, $1\le i\le |\mathcal{S}|$ and $\beta_0:=0$.
\end{proposition}
\begin{proof} Let $0=\lambda^0\subseteq \lambda^1\subseteq\cdots \subseteq\lambda^{\ell(\alpha)}=\lambda$ be the sequence of partitions defining $T\in Tab(\lambda,\alpha)$.  The standardization $\widehat T$ of a SSYT $T\in Tab(\lambda,\alpha)$ is the enumeration of the boxes across the sequence  $0=\lambda^0\subseteq \lambda^1\subseteq\cdots \subseteq\lambda^{\ell(\alpha)}$ defining $T$ where each horizontal strip $\lambda^i/\lambda^{i-1}$ of size $\alpha_i$ is  read SW-NE. This means that $\widehat T$ is a SYT of shape $\lambda$ and its { descent set} $\mathcal{D}(\widehat T)=\{\sum_{j=1}^s\alpha_j: s\in \mathcal{S}\}\subseteq \mathcal{S}(\alpha)$ where $\mathcal{S}$ consists of $s\in\{1,\dots,\ell(\alpha)-1\}$ for which the most SW box in  $\lambda^{s+1}/\lambda^{s}$ appears strictly  below the most NE box in $\lambda^{s}/\lambda^{s-1}$.

  Given $U$  a SYT of shape $\lambda$ and with descent set $\mathcal{S}(\beta)\subseteq \mathcal{S}(\alpha)$  for some composition $\beta$ of $|\lambda|$,
  the standardization may be reversed  to give  a SSYT in $Tab(\lambda,\beta)$. A SYT of shape $\lambda$ with descent set $\mathcal{S}(\beta)$ defines the sequence of partitions $0=\theta^0\subseteq \theta^1\subseteq\cdots \subseteq\theta^{\ell(\beta)}=\lambda$ where each $\theta^j$ consists of the $\beta_1+\cdots+\beta_j$ boxes of $U$  with the entries given by $[ \beta_1+\cdots+\beta_j]$. Therefore
  filling each horizontal strip $\theta^j/\theta^{j-1}$  with $\beta_j $ $j$'s, for all $j\in [ \ell(\beta)]$ gives a SSYT in $Tab(\lambda,\beta)$.  Because  $\mathcal{D}(U)=\mathcal{S}(\beta)\subseteq \mathcal{S}(\alpha)$, given $j\in [ \ell(\beta)]$,  $\beta_j=\alpha_{k+1}+\cdots+\alpha_{k+d}$ for some  $\{k+1,\dots,k+d\}\subseteq [\ell(\alpha)]$. Then we may fill the $\beta_j$ boxes of the horizontal strip $\theta^j/\theta^{j-1}$, from SW-NE, with $\alpha_{k+1}$ $k+1$'s, $\alpha_{k+2}$ $k+2$'s, $\dots$, $\alpha_{k+s}$ $k+d$'s to obtain a SSYT in $Tab(\lambda,\alpha)$.
\end{proof}

\subsection{Dominance order on partitions} \label{subsec:domin} The {\em dominance order} on partitions of the same size $n$, is defined  by setting $\lambda\preceq\mu$ if $|\lambda|=|\mu|=n$ and
$$\lambda_1+\cdots+\lambda_i\leq\mu_1+\cdots+\mu_i,$$
for $i=1,\ldots,\min\{\ell(\lambda),\ell(\mu)\}$. Equivalently,  the Young diagram of $\mu$ is obtained by {\em lifting} at least one box in the Young diagram of $\lambda$. Observe that $\lambda\preceq\mu$ if and only if $\mu'\preceq\lambda'$. The pair
 $(P_n,\preceq)$  with $P_n$ the set of all partitions of $n$ is a lattice with maximum element $(n)$ and minimum element $(1^n)$, and is self dual under the map which sends each partition to its conjugate.
 The interval $[\lambda, \mu]$ in $P_n$ denotes the set of  all partitions $\nu$ such that
$\lambda \preceq \nu \preceq  \mu$.

\begin{remark} \label{re:ineqdom} Note that if $\lambda\preceq \mu$,
 the inequalities $\displaystyle \mu_i\leq \sum_{q= i}^{\ell(\lambda)}\lambda_q=\lambda_i+\sum_{q= i+1}^{\ell(\lambda)}\lambda_q$, for $1\le i\le \ell(\lambda)$, are always satisfied. For $1\le i\le \ell(\lambda)$, either $\mu_i$ is obtained by lifting boxes from  $(\lambda_{i+1},\dots,\lambda_{\ell(\lambda)})$ to $\lambda_i$,  in which case, $\lambda_i\le \mu_i\leq \lambda_i+\sum_{q= i+1}^{\ell(\lambda)}\lambda_q$, or $\mu_i$ is obtained by lifting boxes from $\lambda_i$ to $(\lambda_{1},\dots,\lambda_{i-1})$, in which case, $ \mu_i\leq \lambda_i\leq \lambda_i+\sum_{q= i+1}^{\ell(\lambda)}\lambda_q$.
\end{remark}

 \subsection{The canonical filling  in $Tab(\nu,\alpha)$}
 \label{sec:canonicalfilling}
 Let $\alpha$ be  an arbitrary  composition
and $\nu$  a partition such that $|\nu|=|\alpha|$.
We exhibit a representative element  of $ Tab(\nu,\alpha)$, see also \cite{JaVi}.  The proof provides an $\alpha$-weight canonical filling of a Young diagram of shape $\nu$. The canonical  filling enjoys   descent properties to be used later, see Section \ref{sec:critic} and Proposition \ref{critic}.

\begin{lemma}\label{kostkaproof}
Let $\alpha$ be  any composition and $\nu$ a partition such that $\alpha^+\preceq\nu$. Then,
$K_{\nu,\alpha}>0$
and $Tab(\nu,\alpha)$ has a \textit{canonical} filling representative of the shape $\nu$ with weight $\alpha$. It  is constructed  by filling horizontal strips greedily, from the bottom  to the top of $\nu$, starting with the longest columns, while rows are filled from right to left.
\end{lemma}
\begin{proof}
Assume that $\alpha^+\preceq\nu$.
Then $s:=\ell(\alpha)\geq \ell(\nu)=m$, and the shape $\nu=(\nu_1,\dots,\nu_m)$ has $m$ rows and $\nu_1$ columns. We will show  by induction on $s$ that we can construct a SSYT $T$ of shape $\nu$ and weight $\alpha$ by filling horizontal strips greedily, from bottom  to the top of $\nu$, starting with the longest columns, while rows are filled from right to left. The case $s=1$ is trivial. So, assume $s\ge 2$. If $\alpha=(\alpha_1,\ldots,\alpha_{s-1},\alpha_s^+)$, then fill in, as above, $\alpha_s^+$ entries of the shape $\nu$ with letters $s$. The remaining shape $\widehat{\nu}$ satisfy $(\alpha_1^+,\ldots,\alpha_{s-1}^+)\preceq\widehat{\nu}$ and, by the inductive step, there is a filling as above of the shape $\widehat{\nu}$ with content $(\alpha_1,\ldots,\alpha_{s-1})$. Therefore, there is also a filling of the shape $\nu$ as above, with content $(\alpha_1,\ldots,\alpha_{s-1},\alpha_s^+)$.

Consider now the case $\alpha=(\alpha_1,\ldots,\alpha_s^+,\ldots,\alpha_{s-1},\alpha_i^+)$, where the entry $\alpha_s^+$ is in position $1\leq k<s$. Let $(s;\,s-1)$ and $(k;\,s-1)$ be transpositions of  the symmetric group $S_s$. Write $\widetilde{\alpha}=(s;\,s-1)(k;\,s-1)\alpha=(\alpha_1,\ldots,\alpha_{s-1},\ldots,\alpha_i^+,\alpha_s^+)$. From the previous case, there is a filling of the shape $\nu$ with content $\widetilde{\alpha}$. Consider now the two bottom row strips filled with $\alpha_i^+$ letters $s-1$, and $\alpha_s^+$ letters $s$. We refill these strips first with $\alpha_i^+$ letters $s$, and then with $\alpha_s^+$ letters $s-1$, to obtain a filling of the shape $\nu$ with content $(\alpha_1,\ldots,\alpha_{s-1},\ldots,\alpha_s^+,\alpha_i^+)$. Subtracting the strip filled with $\alpha_i^+$, we get a shape $\widetilde{\nu}$ filled with content $(\alpha_1,\ldots,\alpha_{s-1},\ldots,\alpha_s^+)$. By the inductive hypotheses, it can also be filled in the way described above with content $(\alpha_1,\ldots,\alpha_{s}^+,\ldots,\alpha_{s-1})$. Rejoining the strip $\alpha_i^+$ we get the desired filling.
\end{proof}

\begin{example} Below are examples of SSYT's of partition shape  with  canonical filling:
$$\young(12234,35566,4,5,6)\quad \young(1112222,2334455,445556,55,66)\quad\young(111222333556,3455566,556667,66,77)\quad \young(111222,223344,34555,45,5)$$
\end{example}

The previous lemma gives a constructive proof  of the {\em only if part} of  $(c)$ in the next proposition.
\begin{proposition}\label{kostka}{\em \cite{ful,sa,stanley}}
Let $\alpha$ be  a composition and $\nu$ a partition of $|\alpha|$. Then

$(a)$ $K_{\nu,\alpha^+}=K_{\nu,\alpha}$,

$(b)$ $K_{\alpha^+,\alpha}=1$,

$(c)$ $\alpha^+\preceq\nu\text{ if and only if }K_{\nu,\alpha}>0.$
\end{proposition}
For instance, in \eqref{tab}, $\alpha^+=(4,2,2,1)\preceq \lambda$.

\subsection{Skew-Schur functions, LR tableaux and Littlewood-Richardson rule } \label{subsec:yama}
Let $\Lambda$ denote the ring of symmetric functions in the variables $x = (x_1,x_2,\ldots)$
over $\mathbb{Q}$, say. The Schur functions $s_{\lambda}$ form an orthonormal basis for $\Lambda$, with respect to the Hall inner product, and may be defined in terms
of SSYT by
\begin{equation}\label{schur}
s_{\lambda}=\sum_{T}x^T=\sum_{T}x_1^{t_1}x_2^{t_2}x_3^{t_3}\cdots\in\Lambda,\end{equation}
where the sum is over all SSYT of shape $\lambda$ and  $t_i\ge 0$ is the number of occurrences of $i$ in $T$ \cite{stanley}.
The notion of Schur functions can be generalized to apply to {\em skew shapes}   $\lambda/\mu$.
Replacing $\lambda$ by $\lambda/\mu$ in \eqref{schur} gives the definition
of the skew Schur function $s_{\lambda/\mu}\in\Lambda$ as a sum of monomial weights over all SSYTs of skew shape
$\lambda/\mu$.
 We identify $s_{\lambda/\mu}$ with the skew Schur function indexed by the  skew Young diagram in the { basic form}.

The {\it reading word} $w$ of a SSYT $T$ is the word obtained by reading the entries of $T$ from right to left and top to bottom.
If, for all positive integers $i$ and $j$, the first
$j$ letters of $w$ includes at least as many $i'$s as $(i + 1)'$s, then we say that $w$
is a {\it Yamanouchi word}. Clearly, the content of a Yamanouchi word is a partition.  Yamanouchi words of content $\nu$ are in bijection with  standard Young tableaux of shape  $\nu$ \cite[Section 5.3]{ful}. Each SYT $U$ of shape $\nu$ specifies a Yamanouch word $w_U=w_1\cdots w_{|\nu|}$ of content $\nu$, in the alphabet $[\ell(\nu)]$,   where the number $u\in[|\nu|]$ is in
the $w_u$th row of  the SYT, and this map is  one-to-one. Moreover, one has $w_j\ge w_{j+1}$ unless $j\in \mathcal{D}(U)$ in which case
$w_j<w_{j+1}$.  In \eqref{tab}, for example, \begin{equation}\label{yam}w_{\widehat T}=11\,2211\,32\,3\quad \text{ and} \quad w_{\widehat V}=11\,2211\,332\end{equation} are  Yamanouchi words of content $\nu=(4,3, 2)$, where $\mathcal{D}(\widehat T)=\{2,6,8\}$ and  $\mathcal{D}(\widehat V)=\{2,6\}$.

A  Littlewood--Richardson (LR) tableau  \cite{lr} is a SSYT  whose reading word is Yamanouchi. We denote by $\mathcal{LR}(\lambda/\mu,\nu)$ the set of all LR tableaux of shape $\lambda/\mu$ and content $\nu$. When $\mu$ is empty, $\lambda=\nu$ and the LR tableau of shape $\nu$ and content $\nu$, denoted $Y(\nu)$, is  called the Yamanouchi tableau of shape $\nu$. In fact, $Y(\nu)$ is the unique SSYT of shape and content $\nu$, precisely, the SSYT that is filled with $i$'s in row $i$.
The structure constants $c_{\lambda/\mu}^{\nu}$ in the expansion \eqref{exp} of the skew Schur function $s_{\lambda/\mu}$, in the basis of Schur functions,  are given by the  {\it Littlewood--Richardson
rule} which states that
 the {\it Littlewood--Richardson coefficient} $c_{\lambda/\mu}^{\nu}=\#\mathcal{LR}(\lambda/\mu,\nu)$,   the number of LR tableaux with skew shape $\lambda/\mu$ and content $\nu$ \cite{lr,stanley}.

\subsection{LR tableaux and companion tableaux} \label{subsec:comp} LR tableaux in $\mathcal{LR}(\lambda/\mu,\nu)$ can be replaced by their {\em  companion tableaux}
which are certain SSYTs in $ Tab(\nu,\lambda/\mu)$ whose standardizations encode the Yamanouchi reading words of the LR tableaux in $\mathcal{LR}(\lambda/\mu,\nu)$. Given $G\in Tab(\nu,\lambda/\mu)$,
the containment of the descent set of $\widehat G$ in
$\mathcal{S}(\lambda/\mu)$ guarantees that the filling of $\lambda/\mu$ with Yamanouchi reading word $w_{\widehat G}$ satisfies the row semistandard condition.
Thus any tableau $G\in Tab(\nu,\lambda/\mu)$   specifies through $\widehat G$  a filling of the skew shape  $\lambda/\mu$ with the Yamanouchi reading word $w_{\widehat G}$ of content $\nu$  with the row semistandard condition satisfied but not necessarily the standard condition of the column filling. In addition, by Proposition \ref{prop:desc}, we know that,
a filling of the skew shape $\lambda/\mu$ with a Yamanouchi reading word satisfying the row semistandard condition is encoded by a SYT of shape $\nu$ with descent set in $\mathcal{S}(\lambda/\mu)$. For example, the two Yamanouchi words in \eqref{yam} give  fillings for the skew shape $\lambda/\mu=(2,4,2,1)$ where all satisfy the row semistandard condition. The word $w_{\widehat V}$  does not garantee  the column standard condition in the filling
$$w_{\widehat T}=112211323,\;\young(::::11,:1122,23,3);\; w_{\widehat V}=112211332,\; \young(::::11,:1122,33,2),\; \young(:::::11,::1122,:33,2).$$


Given $H\in \mathcal{LR}(\lambda/\mu,\nu)$ the {\em companion tableau} $G$ of $H$ is the  SSYT in $ Tab(\nu,\lambda/\mu)$  whose  $\nu_i$ entries of each row $i$ of $G$  are the numbers of the rows   of $H$ where the $\nu_i$  $i$'s are filled in. This defines  a bijection between $\mathcal{LR}(\lambda/\mu,\nu)$ and a subset $\rm LR_{\nu,\lambda/\mu}$, of $ Tab(\nu,\lambda/\mu)$ that sends $H\in \mathcal{LR}(\lambda/\mu,\nu)$ to $G\in \rm LR_{\nu,\lambda/\mu}$.  Therefore, the LR coefficient in \eqref{exp} also satisfies
\begin{equation}c_{\lambda/\mu}^{\nu}=\#\mathcal{LR}(\lambda/\mu,\nu)=\#\rm LR_{\nu,\lambda/\mu}.\label{lr}\end{equation}

The set $\rm LR_{\nu,\lambda/\mu}$ may be characterized in several ways: by linear inequalities as in \cite{GZpolyed};  or observing that $ Tab(\nu,\lambda/\mu)$ is a subset of the $gl_n$-crystal $B(\nu)$ consisting of all SSYTs of shape $\nu$ in the alphabet $[n]:=\{1,\dots,n\}$, $n\ge \ell(\lambda)$, \cite{kwon,bumschi}. The highest weight element of $B(\nu)$ is $Y(\nu)$ and  $\rm LR_{\nu,\lambda/\mu}$ consists of the vertices $G$ in $B(\nu)$ such that  $Y(\mu)\otimes G$ is a highest weight element of weight $\lambda$ of $B(\mu)\otimes B(\nu)$ \cite[Section 4.3]{kwon}.

Given  $G\in Tab(\nu,\alpha)$, for each $1\leq i\leq \ell(\nu)$, and $j\geq i$, let $\chi_j^i$ denote
the multiplicity of letter $j$ in row $i$ of $G$. Note that, for $j=1,\ldots,\ell(\alpha)$, $\chi_j^i=0$, whenever $1\le j<i$.
 Fix $\mu\subseteq \lambda$ so that $\alpha=\lambda/\mu$. One  then has the bijection,
\begin{equation}\label{comp}
\phi_{\lambda/\mu}:\rm LR_{\nu,\lambda/\mu}=\{\text {$G\in Tab(\nu,\alpha)$:  $Y(\mu)\otimes G\approx_{gl_n} Y(\lambda)$ }\}\longrightarrow \mathcal{LR}(\lambda/\mu,\nu),\;G\mapsto\phi_{\lambda/\mu}(G),
\end{equation}
such that $\phi_{\lambda/\mu}(G)$ is the $\nu$-weight semistandard filling of $\lambda/\mu$  by putting $\chi_j^i$ letters  $i$, starting from the left, in  row $j$ of the skew-shape $\lambda/\mu$, 
for $i=1,\ldots,\ell(\nu)$, and $j=1,\dots,\ell(\alpha)$. The reading word of $\phi_{\lambda/\mu}(G)$ is precisely the Yamanouchi word  of weight $\nu$,  $w_{\widehat G}=w_1\cdots
w_{\alpha_1}w_{\alpha_1+1}$ $\cdots w_{\alpha_1+\alpha_2}w_{\alpha_1+\alpha_2+1}\cdots w_{\alpha_1+\cdots+\alpha_{\ell(\alpha)-1}}\cdots w_{|\alpha|}$.
That is, $\rm LR_{\nu,\lambda/\mu}$ consists of those tableaux in $Tab(\nu,\alpha)$  assigning to the skew shape  $\lambda/\mu$  a semistandard filling of content $\nu$ whose reading word is the Yamanouchi $w_{\widehat G}$ (hence an LR filling).
 Theorem \ref{th:gessel} characterises $\rm LR_{\nu,R_\alpha}$ in the case of connected ribbons $R_\alpha$.

\subsection{ Schur support and symmetries}
The definition \eqref{interval} of Schur support of the skew shape  $\lambda/\mu$  can be rephrased as follows: $\nu\in[\lambda/\mu]$ if and only if $\mathcal{LR}(\lambda/\mu,\nu)\neq\emptyset$, equivalently, ${\rm LR}_{\nu,\lambda/\mu}\neq\emptyset$.

 LR coefficients satisfy a number of symmetries \cite{stanley,acm,akt16}, including:
$c_{\lambda/\mu}^{\nu}=c_{\lambda/\nu}^{\mu}$, $c_{\lambda/\mu}^{\nu}=c_{{(\lambda/\mu)}^{\circ}}^{\nu}$  where $(\lambda/\mu)^\circ$ is the $\pi$-rotation of $\lambda/\mu$, and $c_{\lambda/\mu}^{\nu}=c_{\lambda'/\mu'}^{\nu'}$.
As  a consequence  $[\lambda/\mu]=[(\lambda/\mu)^\circ]$ and $[(\lambda/\mu)']=[\lambda/\mu]'$ where
$$ s_{\lambda/\mu}=s_{(\lambda/\mu)^{\circ}}\;\;\text{and}\;\; s_{\lambda'/\mu'}=\sum_{\nu\in [\rm r(\lambda/\mu),\rm c(\lambda/\mu)']} c_{\lambda/\mu }^{\nu} s_{\nu'}.$$
The full support of one of the shapes  $\lambda/\mu$, $(\lambda/\mu)'$ or $(\lambda/\mu)^\circ $ implies the full support of any of the others.
When $\lambda/\mu$ is not connected, and consists of two connected components $A$ and $B$,  and may themselves be either Young diagrams or skew Young diagrams, then the  combinatorial definition of (skew) Schur function \eqref{schur} gives \cite{stanley}
$s_{\lambda/\mu}=s_{A}s_{B}=s_Bs_A.$ This means that a skew Schur function is invariant under permutation and rotation of the connected components.



\section{Ribbons} \label{sec:ribbon}
A {\em ribbon} is a skew shape which does not contain a $2\times 2$ block as a subdiagram and it is connected when each pair of consecutive rows intersects in exactly one column. Thus, any composition $\alpha=(\alpha_1,\dots,\alpha_{\ell(\alpha)})$ determines a unique connected ribbon consisting of $\ell(\alpha)$ rows (or {\em parts}) $<\alpha_i>$ of length $\alpha_i$, for $i=1,\ldots,\ell(\alpha)$,
from top to bottom.

Given the composition $\alpha$, $R_\alpha$ will denote a ribbon (not necessarily connected) where row lengths
from top to bottom are given by the parts of $\alpha$ and  adjacent rows  overlap in at most  one column.
If each row is at least two boxes in length  then the column length is at most two otherwise the column length might be bigger than two.
If $\beta$ is another composition, the direct sum $R_\alpha\oplus R_\beta$ of the ribbons $R_\alpha$ and $R_\beta$,  is the ribbon $R_{\alpha\cdot\beta}$ where the ribbons $R_\alpha$ and $R_\beta$ have no edge in common. In general,
 $R_\alpha$ is a direct sum of connected ribbons unless otherwise stated.

.
\subsection{Overlapping partition of a  ribbon with parts at least two}
\label{subsec:overlap}
 In this subsection, we only consider compositions $\alpha$ with parts $\ge 2$, and therefore the ribbon $R_\alpha$  has columns of length at most two.
\begin{definition} \label{overlap} Let $\alpha$ be an arbitrary composition with parts $\ge 2$.
 The \textit{ overlapping partition} of $R_\alpha$ is the partition $p=(p_1, p_2,\ldots, p_{\ell(\alpha)-1}, 0)$, $\ell(p)\le\ell(\alpha)-1$, such that
 $p_i$ is
 the number of columns of length two among the smallest $\ell(\alpha)-i+1$ rows of $R_\alpha$ in lowest position, for $ i=1,\dots, \ell(\alpha)$. When $\alpha$ is a partition, $p_i$ is the number of columns of length two in the last $\ell(\alpha)-i+1$ rows of $R_\alpha$ for $ i=1,\dots, \ell(\alpha)$.
\end{definition}

Observe that $\displaystyle\sum_{j=i}^{\ell(\alpha)}\alpha_j^+-p_i$ is the number of columns of $\displaystyle R_\alpha\setminus(\bigcup_{j=1}^{i-1}<\alpha^+_j>)$, for  $1\le i\le \ell(\alpha)$. In particular,
 $|\alpha|-p_1$ is the number of columns of $R_\alpha$ and thus the Schur interval of  a ribbon $R_{\alpha}$ with overlapping partition $p$ is $[\alpha^+,(|\alpha|-p_1,p_1)]$.  When $\alpha$ is a partition, one obtains \eqref{partsupp} as  a special case of this interval.

\begin{proposition} Let $\alpha$ be a composition with  parts $\ge 2$. For $1\le i\le \ell(\alpha)$, let $k_i\in \{1,\dots,\ell(\alpha)\}$ be the number of connected components (ribbons) of $\displaystyle R_\alpha\setminus(\displaystyle\bigcup_{j=1}^{i-1}<\alpha^+_j>)$. Then $p_i=\ell(\alpha)-(i-1)-k_i$, for $i=1,\dots,\ell(\alpha)-1$, with $0\le p_{\ell(\alpha)-1}\le 1$, and
\begin{equation}\label{overlap}p\subseteq (\ell(\alpha)-1,\dots,2,1,0)\subseteq (|\alpha|-\alpha^+_1=\sum_{j=2}^{\ell(\alpha)}\alpha_j^+,\dots,|\alpha|-\sum_{j=1}^{\ell(\alpha)-1}
\alpha^+_i=\alpha^+_\ell(\alpha),0),\end{equation}
where the set of distinct entries of $p$ is contained in  $\{\ell(\alpha)-1,\ell(\alpha)-2,\dots,2,1,0\}$.
\end{proposition}

\begin{proof}
Observe that, $p_1=\ell(\alpha)-k_1\in\{0,1,2,\dots,\ell(\alpha)-1\}$ and by induction on $i\ge 1$, $p_i=\ell(\alpha)-(i-1)-k_i$  is the first entry   of the overlapping partition of $\displaystyle R_\alpha\setminus(\displaystyle\bigcup_{j=1}^{i-1}<\alpha^+_j>)$, $1\le i\le \ell(\alpha)$.
 Henceforth  $\displaystyle 0\le p_{i+1}\le p_i\le \ell(\alpha)-i\le \sum_{j=i+1}^s\alpha_j^+$, for $i=1,\ldots, \ell(\alpha)-1$.
\end{proof}

A ribbon $R_\alpha$ is connected if and only if $p_1=\ell(\alpha)-1$,  otherwise $p_1\in\{0,1,2,\dots,\ell(\alpha)-2\}$.
It is an horizontal strip  if $p_1=0$.
When $\alpha=\alpha^+$, a ribbon $R_{\alpha^+}$ (not necessarily connected) is uniquely defined by the partition $\alpha$ and its overlapping partition $p$ and hence  $R_{\alpha}^p$ denotes such ribbon. In fact, more can be said.
 It is shown next that  monotone ribbons with at least  two boxes in each row are in bijection with  pairs of partitions $(\alpha,p)$ where the parts of $p$ are assigned  by  a multiset of $\{\ell(\alpha)-k,\dots,1\}$ of cardinality $\ell(\alpha)-k\le \ell(p)<\ell(\alpha)$ with $k\in\{1,\dots,\ell(\alpha)\}$.  Recall Example \ref{example:first}, $R_{(3)}\oplus R_{(3,2)}\oplus R_{(2,2)}=R^{(2,2,1,1,0)}_{(3,3,2,2,2)}$ is defined by the partition pair $(3,3,2,2,2)$ and  $p=(2,2,1,1,0)$.
\begin{proposition} \label{prop:biject} Let $\alpha $ be a partition with parts $\ge 2$ and  let $ k\in\{1,\dots,\ell(\alpha)\}$.
There is a bijection between ribbons $R_\alpha$ with $k$ connected components and  multisets of $\{\ell(\alpha)-k, \dots,2,1\}$ of cardinality $\ell(\alpha)-k\le \ell(p)\le \ell(\alpha)-1$ assigning the parts of the overlapping partition $p$.
\end{proposition}

\begin{proof}
Let $R_\alpha$ with $k$ connected components. One has $p_1=\ell(\alpha)-k$, and, for $2\le i\le\ell(\alpha)$,  $p_{i}=p_{i-1}$ if rows $i$ and $i-1$ of $R_\alpha$ do not overlap, and  $p_i=p_{i-1}-1$ otherwise. In particular, $0\le p_{\ell(\alpha)-1}\le 1$. Henceforth the parts of $p$ form a multiset of $\{\ell(\alpha)-k, \ell(\alpha)-k-1,\dots,2,1\}$  of cardinality $\ell(\alpha)-k\le \ell(p)\le \ell(\alpha)-1$.
  Let $R_\alpha$ and $\tilde R_\alpha$ be two distinct ribbons (skew shapes do not coincide) with $k$ connected components and overlapping partitions $p$ and $\tilde  p$ respectively. Let us choose the first $i\in\{2,\dots,\ell(\alpha)\}$ such that  rows $i$ and $i-1$ in  one of them overlap and in the other do not. Then $p_q=\tilde p_q$, for $1\le q\le i-1$, and $p_i=p_{i-1}-1$ and $\tilde p_i= p_{i-1}$ or reciprocally, and thus $p\neq \tilde  p$.

Let us consider  a multiset of $\{\ell(\alpha)-k, \ell(\alpha)-k-1,\dots,2,1\}$  of cardinality $\ell(\alpha)-k\le \ell(p)\le \ell(\alpha)-1$, and $ p=(p_1, p_2,\ldots, p_{\ell(p)}, 0^{\ell(\alpha)-\ell(p)})$  the partition where $\{p_1, p_2,\ldots, p_{\ell(p)}=1\}$ is the given multiset. We have to construct a ribbon $R_\alpha$ with $k$ components and overlapping partition $p$. Put the last $\ell(\alpha)-\ell(p)$ rows of $R_\alpha$ pairwise disconnected and, observing that $p_{\ell(p)}=1$, whenever $p_{i+1}=p_{i}$  rows $i$ and $i+1$ of $R_\alpha$ do not overlap, and  $p_{i+1}=p_{i}-1$ otherwise for $1\le i\le \ell(p)$.
\end{proof}

\begin{remark} Observe that if $\alpha$ is not a partition, in general $\alpha$ and $p$ do not uniquely define a disconnected ribbon with more than two connected components. For instance, below $\alpha=(2).(3,2).(3,2)=(2).(3,2,3).(2)$, and $R_{(2)}\oplus R_{(3,2)}\oplus R_{(3,2)}$, $R_{(2)}\oplus R_{(3,2,3)}\oplus R_{(2)}$ are distinct  ribbons with the same overlapping partition $p=(2,1,0,0,0)$,
$$R_{(2).(3,2).(3,2)}=\young(::::::::\hfill\hfill,:::::\hfill\hfill\hfill,::::\hfil\hfill,:\hfil\hfill\hfill,\hfill\hfill)\;
R_{(2).(3,2,3).(2)}=\young(::::::::\hfill\hfill,:::::\hfill\hfill\hfill,::::\hfill\hfill,::\hfil\hfill\hfill,\hfill\hfill).$$
\end{remark}

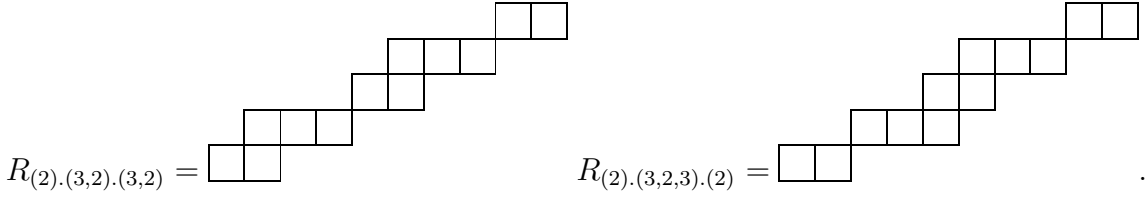
\begin{example}\label{partitionribbon}
$(a)$ Ribbons with shape $R_{\alpha}^p$, for $\alpha=(4,4,3,2)=\alpha^+$, $\ell(\alpha)=4$:
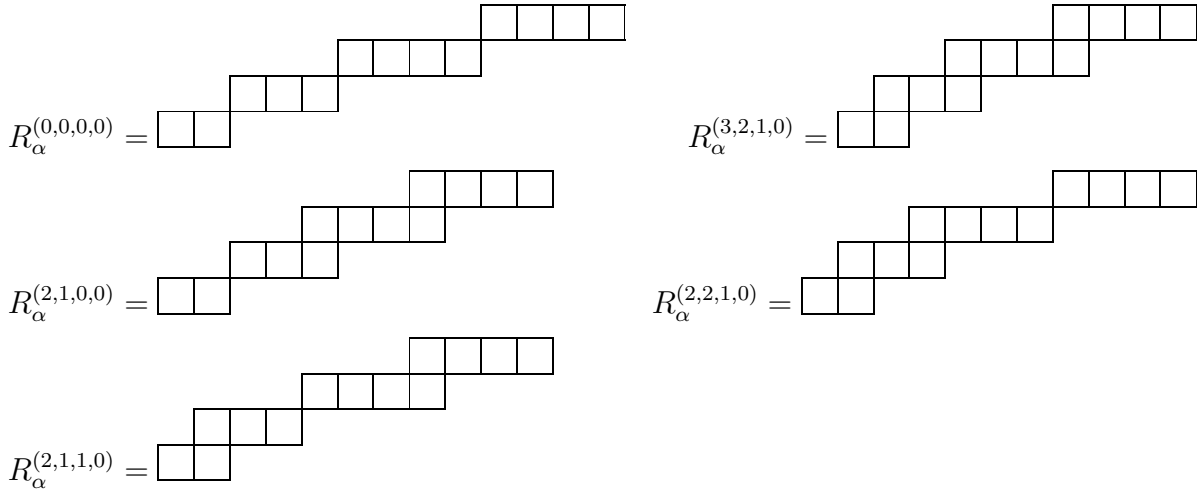
\begin{align*}
&R_{\alpha}^{(0,0,0,0)}=\young(:::::::::\hfill\hfill\hfill\hfill,:::::\hfill\hfill\hfill\hfill,::\hfill\hfill\hfill,\hfill\hfill)&
R_{\alpha}^{(3,2,1,0)}=\young(::::::\hfill\hfill\hfill\hfill,:::\hfill\hfill\hfill\hfill,:\hfill\hfill\hfill,\hfill\hfill)\\
&R_{\alpha}^{(2,1,0,0)}=\young(:::::::\hfill\hfill\hfill\hfill,::::\hfill\hfill\hfill\hfill,::\hfill\hfill\hfill,\hfill\hfill)&
R_{\alpha}^{(2,2,1,0)}=\young(:::::::\hfill\hfill\hfill\hfill,:::\hfill\hfill\hfill\hfill,:\hfill\hfill\hfill,\hfill\hfill)\\
&R_{\alpha}^{(2,1,1,0)}=\young(:::::::\hfill\hfill\hfill\hfill,::::\hfill\hfill\hfill\hfill,:\hfill\hfill\hfill,\hfill\hfill)
\end{align*}
 $(b)$ $R_\alpha= R_{(2,3,2,3)}$  with  $\ell(\alpha)=4$ and $p=(3,1,0,0)\subseteq (3,2,1,0)$. The sequence of ribbons $R_\alpha\setminus(\cup_{j=1}^{i-1}<\alpha^+_j>)$, $1\le i\le \ell(\alpha)$, is depicted below

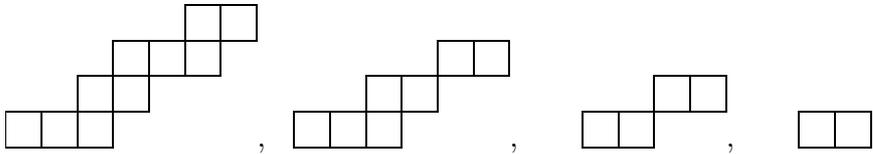
\young(:::::\hfill\hfill,:::\hfill\hfill\hfill,::\hfill\hfill,\hfill\hfill\hfill),\;\;
\young(::::\hfill\hfill,::\hfil\hfill,\hfill\hfill\hfill),\;\;
\young(:::\hfill\hfill,:\hfill\hfill),\;\;
\young(:\hfill\hfill).

\end{example}


\subsection{ LR ribbons and companion tableaux}
\label{subsec:ribcomp} Let $\alpha$ be an arbitrary composition. As we have seen in \eqref{comp}, if one picks $G\in Tab(\nu,\alpha)$ to be the companion tableau of some LR ribbon in $\mathcal{LR}(R_\alpha,\nu)$ the Yamanouchi word $w_{\widehat G}$ has to guarantee in the filling of $R_\alpha$  the standard condition  in the columns. The overlapping of two consecutive rows reduces to at most  one column.
Thus for ribbon shapes $R_\alpha$ one has just to avoid the violation of the standard condition on the   overlapping row pairs which  just occurs in one column. In other words, whenever, in $R_\alpha$, rows $\alpha_k$ and $\alpha_{k+1}$ overlap then in the reading word $w_{\widehat G}$ the subword $w_{\alpha_1+\cdots+\alpha_{k}}w_{\alpha_1+\cdots+\alpha_k+1}$ is strictly increasing which means $\alpha_1+\cdots+\alpha_{k}$ is a descent of $\widehat G$.
In the case of connected ribbons $R_\alpha$, this is exactly the content of
 Theorem \ref{th:gessel}: to
 avoiding the violation of the semistandard condition  on the  overlapping row pairs it    requires  the descent set of the standardization of the companion tableau  to be equal to $\mathcal{S}(\alpha)$.
To figure out what are the conditions to be imposed on the entries of a SSYT to be the  companion of an LR ribbon, we
take into account  the bijection \eqref{comp}, whose domain we now extend  to the set $Tab(\nu,\alpha)$. Thanks to Proposition \ref{prop:desc} we may define the bijection
\begin{definition}
 Let $\nu$ be a partition and $\alpha$  an arbitrary  composition of $|\nu|$. Given  $T\in Tab(\nu,\alpha)$, for each $1\leq i\leq \ell(\alpha)$, and $j\geq i$, let $\chi_j^i$ denote
the multiplicity of letter $j$ in row $i$ of $T$.
Given a ribbon $R_\alpha$, define the map $$\varphi_{R_\alpha}: Tab(\nu,\alpha)\longrightarrow \{\text{$\nu$-Yamanouchi fillings of } R_{\alpha}\text{ with row semistandard condition} \},$$
 such that $\varphi_{R_\alpha}(T)$ is the filling of $R_{\alpha}$  by putting $\chi_j^i$ letters  $i$ in each row strip $<\alpha_j>$,  starting from the left, for $i=1,\ldots,\ell(\nu)$, and $j=1,\dots,\ell(\alpha)$, that is, the reading word of $\varphi_{R_\alpha}(T)$ is $ w_{\widehat T}$.
\end{definition}

\begin{remark}\label{l1}   When $R_\alpha$ is an horizontal strip, the map  $\varphi_{ R_\alpha}: Tab(\nu,\alpha)\longrightarrow \mathcal{LR}(R_{\alpha},\nu)$ is a bijection and $c^\nu_{R_\alpha}=K_{\nu,\alpha}$.
\end{remark}


\begin{example} \label{varphi}  Let $\nu=(6,4,2)$ and  $\alpha=(4,2,2,2,2)$.

$(a)$ Let $T=\young(111124,2335,45)\in Tab(\nu,\alpha)$
with $\mathcal{D}(\widehat T)=\mathcal{S}(\alpha)$ and
$
\chi_1^1=4, \chi^1_2=1,\chi^1_3=0,\chi^1_4=1, \chi^1_5=0,
 \chi^2_2=1,\chi^2_3=2, \chi^2_4=0,  \chi^2_5=1,
 \chi^3_3=0, \chi^3_4=1,\quad \chi^3_5=1.
$.
Considering  the overlapping sequence  $p=(4,3,2,1,0)$ for $\alpha$, we get the tableau
$$\varphi_{R_\alpha^p}(T)=\young(::::1111,:::12,::22,:13,23).$$
 with  Yamanouchi reading word $w_{\widehat T}$  satisfying both requirements of  semistandard property.
Thus, $\varphi_{R_\alpha^p}(T)\in \mathcal{LR}(R_{\alpha^p},\nu),$ and $T$ is the companion tableau of $\varphi_{R_\alpha^p}(T)$.

$(b)$ Next, one exhibits the violation of the column semistandard condition of $\varphi_{R_\alpha^p}(T)$ in the two possible ways.
 Consider now $Q=\young(111122,3345,45)$ and $ V=\young(111133,2245,45)$ in $Tab(\nu,\alpha)$ where $\mathcal{S}(\alpha)=\{4,6,8,10\}$, $\mathcal{D}(\widehat Q)=\{6,8,10\}=\mathcal{S}(\alpha)\setminus\{4\}$, $w_{\widehat Q}=1111\,11223232$, $w_4=w_5$, and   $\mathcal{D}(\widehat V)=\{4,8,10\}=\mathcal{S}(\alpha)\setminus\{6\}$, $w_{\widehat V}=111122\,113232$, $w_6>w_7$.

  If $p=(4,3,2,1,0)$,  the strict increasing filling along columns of $\varphi_{R_\alpha^p}(Q)$ and $\varphi_{R_\alpha^p}(V)$ fails in the overlapping of the rows
$<\alpha_1>$ and $<\alpha_2>$, and  $<\alpha_2>$ and $<\alpha_3>$, respectively:
$$\varphi_{R_\alpha^p}(Q)=\young(::::1111,:::11,::22,:23,23),\;\varphi_{R_\alpha^p}(V)=\young(::::1111,:::22,::11,:23,23)\not\in \mathcal{LR}(R^{ p}_{\alpha},\nu).$$
$(i)$ In the first case, $w_{\alpha_1}=w_{\alpha_1+1}$, if   we instead consider the overlapping sequence $\tilde p=(3,3,2,1,0)$, then $Q$ becomes the companion tableau of $\varphi_{R_\alpha^{\tilde p}}(Q)\in \mathcal{LR}(R^{\tilde p}_{\alpha},\nu)$.

$(ii)$ In the second case, $w_{\alpha_1+\alpha_2}>w_{\alpha_1+\alpha_2+1}$, we keep $p$  but change $V$ to
$U=\young(111123,2345,45)$, where $\mathcal{D}(\widehat U)=\{4,6,8,10\}$, then $U$ is the companion tableau of $\varphi_{R_\alpha^p}(U)\in\mathcal{LR}(R^{ p}_{\alpha},\nu)$,
$$\varphi_{R_\alpha^{\tilde p}}(Q)=\young(:::::1111,:::11,::22,:23,23),\;\;\varphi_{R_\alpha^p}(U)=\young(::::1111,:::12,::12,:23,23).$$
\end{example}

\begin{example} \label{standard} $(a)$ Let $\alpha=(3,3,2,3,3)$ and $\nu=(4,4,1)$.
 Let
$$Q=\young(1122,3445,5),\; V=\young(1234,5679,8)\in Tab(\nu,\alpha).$$
The descent set   $\mathcal{D}(\widehat Q)=\{\alpha_1+\alpha_2=4,\alpha_1+\alpha_2+\alpha_3+\alpha_4=7\}= S(\alpha)\setminus\{\alpha_1,
\alpha_1+\alpha_2+\alpha_3\}$; and  the descent set   $\mathcal{D}(\widehat V)=\mathcal{S}(\alpha)$.
 The tableaux $Q$ and $V$ are companion tableaux of the following  LR fillings for $R_{(2).(2,1).(2,2)}=R_{(2)}\oplus R_{(2,1)}\oplus R_{(2,2)}$,
$$\varphi_{R_\alpha}(Q)=\young(:::::11,:::11,:::2,:22,23)\quad \varphi_{R_\alpha}(V)=\young(:::::11,:::12,:::2,:13,22).$$
\end{example}

From Proposition \ref{prop:desc} we easily conclude

\begin{proposition}\label{l2new}
 Let $G\in Tab(\nu,\alpha)$ and  $R_\alpha$ a ribbon. Then

$(a)$ $\varphi_{R_\alpha}(G)\in \mathcal{LR}(R_\alpha,\nu)$   if and only  if whenever two consecutive rows $j$ and ${j+1}$ of $R_\alpha$ overlap then $\sum_{k=1}^j\alpha_k$ is  in  the  descent set of $\widehat G$.

$(b)$ if $R_\alpha$ is connected,  $\varphi_{R_\alpha}(G)$ is  an  LR ribbon if and only if  $\mathcal{S}(\alpha) =\mathcal{D}(\widehat G)$.

$(c)$ if $\displaystyle R_\alpha=\oplus_{i=1}^{k}R_{\tilde\alpha_i}$ is  a direct sum of $k$ connected  ribbons, $\varphi_{R_\alpha}(G)$ is  an  LR ribbon if and only if  $\mathcal{S}(\alpha)\setminus \{\sum_{i=1}^r|\tilde\alpha_i|, 1\le r\le k\}\subseteq \mathcal{D}(\widehat G)$.
\end{proposition}

\begin{corollary}\label{l3new}
$(a)$  Let $R_\alpha$ be a connected ribbon and $\nu$ a partition  such that $|\nu|=|\alpha|$.
Then
 \begin{enumerate}
\item  $\rm LR_{\nu,R_\alpha}=\{G\in Tab(\nu,\alpha): \text{ $\mathcal{S}(\alpha)=\mathcal{D}(\widehat G)$ }\}$.
 \item $c_{R_\alpha}^\nu=d_{\nu,\alpha}$ the number of standard Young tableaux of shape $\nu$ with descent set $\mathcal{S}(\alpha)$.
 \end{enumerate}

$(b)$ Let $\alpha=\tilde\alpha_1\cdot\cdots\cdot \tilde\alpha_s$ and $\displaystyle R_\alpha=\oplus_{i=1}^{k}R_{\tilde\alpha_i}$  a direct sum of $k$ connected  ribbons $R_{\tilde\alpha_i}$. 
Then
 \begin{enumerate}
\item
$\rm LR_{\nu,R_\alpha}=\{G\in Tab(\nu,\alpha): \text{ $\mathcal{S}(\alpha)\setminus \{\sum_{i=1}^r|\tilde\alpha_i|, 1\le r\le k\} \subseteq  \mathcal{D}(\widehat G)$ }\}$.
 \item $c_{R_\alpha}^\nu$ is the number of standard Young tableaux of shape $\nu$ whose descent set,  a subset of $S(\alpha)$, contains $\mathcal{S}(\alpha)\setminus \{\sum_{i=1}^r|\tilde\alpha_i|, 1\le r\le k\}$.
 \end{enumerate}
\end{corollary}



\subsection{ The critical  set of a SSYT in $Tab(\nu,\alpha)$} \label{sec:critic}
We now reduce our study to compositions $\alpha$ with parts $\ge 2$.
 Given $T\in Tab(\nu,\alpha)$, recall that $\mathcal{D}(\widehat T)\subseteq \mathcal{S}(\alpha)$. The goal is to identify in the SSYT $T$ the entries of $\widehat T$ that are elements of $\mathcal{S}(\alpha)\setminus \mathcal{D}(\widehat T)$. More precisely, the numbers $j\in \{2,\dots,\ell(\alpha)\}$ in the filling of $T$ such that in the word $ w_{\widehat T}=w_1\cdots w_{\ell(\alpha)}$ (Subsection \ref{subsec:yama}) either it occurs $(1)$ $w_{\sum_{k=1}^{j-1}\alpha_k}=w_{1+\sum_{k=1}^{j-1}\alpha_k}$; or $(2)$ $w_{\sum_{k=1}^{j-1}\alpha_k}>w_{1+\sum_{k=1}^{j-1}\alpha_k}$.
  See Example \ref{varphi} $(b)$, $(i)$.

   The serious rejection for $T\in Tab(\nu,\alpha)$ to be a companion tableau of a LR ribbon in $\mathcal{LR}(R_\alpha,\nu)$ occurs when one has repeated letters in a  column of length $2$ of $\varphi_{R_\alpha}(T)$.
  This means that we are collecting in the filling of $T$  the numbers $j\in \{2,\dots,\ell(\alpha)\}$  verifying $(1)$. This numbers define a subset  of $\{2,\dots,\ell(\alpha)\}$  called
  the {\em critical  set  $\mathcal{C}(T)$ of $T$}.
   The
 set $\mathcal{C}(T)$ of critical numbers of $T$ verifies   $$\mathcal{C}(T)=\{j\in \{2,\dots,\ell(\alpha)\}: \text{$\sum_{k=1}^{j-1}\alpha_k$ and $1+\sum_{k=1}^{j-1}\alpha_k$  are entries in a same row of $\widehat T$}\}$$
 $$=\{j\in \{2,\dots,\ell(\alpha)\}: w_{\widehat T}=w_1\cdots w_{\ell(\alpha)}\; \text{and} \; w_{\sum_{k=1}^{j-1}\alpha_k}=w_{1+\sum_{k=1}^{j-1}\alpha_k}\}$$
 $$\subseteq\{j\in \{2,\dots,\ell(\alpha)\}:\sum_{k=1}^{j-1}\alpha_k\in \mathcal{S}(\alpha)\setminus \mathcal{D}(\widehat T)\}.$$

From Proposition \ref{l2new} and Corollary \ref{l3new}, we conclude that $\mathcal{C}(T)$ detects the elements $j$ in the alphabet $\{2,\dots,\ell(\alpha)\}$ for which $\sum_{k=1}^{j-1}\alpha_k\in S(\alpha)$ are not in the descent set of $\widehat T$ and give rise in $\varphi_{R_\alpha}(T)$ to a filling of a column of length $2$ with two repeated letters. This  column  of length two is obtained in the  overlapping of the rows $j-1$ and $j$ of $R_\alpha$  and is filled with a  same letter $i< j$.
Henceforth, because $T$ is a sequence of partitions   $0=\lambda^0\subseteq \lambda^1\subseteq\cdots \subseteq\lambda^{\ell(\alpha)}$,  the right most box of the horizontal strip $\lambda^{j-1}/\lambda^{j-2}$    is glued with the left most box of $\lambda^{j}/\lambda^{j-1}$, and one has
  \begin{proposition}\label{critical} Let $T\in Tab(\nu,\alpha)$ and $j\in\{2,\dots,\ell(\alpha)\}$. The number $j \in \mathcal{C}(T)$ or  $j$ is a critical number of $T$,  if for some $i\in\{1,\dots,j-1\}$,
 $\chi_{j-1}^i,\chi_{j}^i\neq 0$, and $\chi^k_{j-1}=\chi^h_{j}=0$ for all $k\leq i-1$ and $h\geq i+1$.
In this case, we also say that the integer $j$  generates the critical row $i$ of $T$.
\end{proposition}

  The numbers in $\mathcal{S}(\alpha)\setminus \mathcal{D}(\widehat T)$ giving rise to the violation $(2)$  by inverting the increasing order in the filling of a column of length two in $\varphi_{R_\alpha}(T)$ are
negligible critical numbers,  because they may be removed anytime  without creating new  ones.  In the SSYT $T$ we collect the numbers $j\in\{2,\dots,\ell(\alpha)\}$ verifying condition $(2)$.
In fact, if, in a such column of $\varphi_{R_\alpha}(T)$, resulting from the overlapping of rows, say,  $j-1$ and $j$ of $R_\alpha$, one has $\young(:::bx\cdots y,w\cdots za)$, with $x\ge b>a\ge z$, we may easily correct this Yamanouchi filling, without creating new violations in the new Yamanouchi filling, by just reordering the entries of that column,
$\young(:::ax\cdots y,w\cdots zb)$, with $y\ge\cdots\ge x>a< b> z\ge\cdots\ge w$, to obtain an  LR ribbon.  This  tells that $j-1$ appears in $T$ only in row $b$ and possibly below, and $j$ only appears  in row $a$ and above. (The horizontal strip $\lambda^{j-1}/\lambda^{j-2} $ is strictly below the horizontal strip $\lambda^{j}/\lambda^{j-1} $.) Henceforth, we should  replace in row $a$ of $T$ the left most entry $j$ with $j-1$, and replace in row $b$ of $T$  the rightmost entry $j-1$ with $j$. One then says  $j$ is a negligible critical number of $T$. See Example \ref{varphi}, $(b)$, $(ii)$.

Canonical fillings of SSYTs do not have  negligible critical numbers and the critical numbers have an easier formulation. Note that the multiplicity of letter $j\ge i$ in row $i$ of $T\in Tab(\nu,\alpha)$ satisfies $\chi_{j}^i\le \alpha_{j}$.

\begin{proposition} \label{critic} Let $T\in Tab(\nu,\alpha)$ with canonical filling and $j+1\in  \{2,\dots,\ell(\alpha)\}$. Then
 $j+1$ is a critical number  of $T$ if and only if  $\chi_{j}^i=\alpha_{j}$ and $\chi_{j+1}^i=\alpha_{j+1}$ for some $i\in\{1,\dots,j\}$.
\end{proposition}

\begin{proof} Recall $\ell (\nu)\le \ell(\alpha)$. If $T$ has canonical filling and $\chi^i_j,\chi^i_{j+1}\neq 0$  with $\chi^h_{j+1}=0$ for all $h\geq i+1$, then below row $i$ the entries are empty or bigger than $j+1$. Therefore there is no need to put $j+1$'s in rows above $i$ because positions of row $i$ have been used to put the letter $j$, that is, one  also has $\chi^h_{j+1}=0$ for all $h<i$.  Hence $\chi_{j+1}^i=\alpha_{j+1}$. Similarly $\chi^h_{j}=0$ for all $h>i$ because $j+1$ has to be filled first and there are no $j+1$ below row $i$. Hence $\chi_{j}^i=\alpha_{j}$.
\end{proof}
We then may conclude
\begin{proposition}\label{l2} Let $T\in Tab(\nu,\alpha)$ without negligible critical numbers. Then   $T\not\in\rm LR_{\nu, R_\alpha}$ if and only if $T$ has a critical number  $j+1\in\{2,\dots,\ell(\alpha)\}$ such that rows $j$ and $j+1$ of $R_{\alpha}$ overlap. In this case, the   column  of length two obtained in the  overlapping of rows $<\alpha_j>$ and $<\alpha_{j+1}>$ of $R_\alpha$  is filled with a  same letter $i< j+1$.
\end{proposition}

\subsection{Effectiveness of critical numbers} \label{sec:efectcritic} The ribbon $R_\alpha$, with rows of length at least two, is now assumed to be  connected or monotone up to a permutation and rotation of the connected components of $R_\alpha$.
Since the ribbon can be monotone and disconnected, the overlapping partition $p$ is used to detect the effectiveness of the critical numbers of a companion tableau in $\rm LR_{\nu,R_\alpha^p}$.
\begin{definition}\label{p}
Let $T\in Tab(\nu,\alpha)$ and let $p$ be an overlapping partition for $\alpha$. A critical  number $j$ of $T$ is said to be $p$-effective if rows $j-1$ and $j$ of $R_{\alpha}^p$ overlap. Otherwise, the critical  number $j$ is said to be $p$-ineffective.
\end{definition}

 This is a reformulation of Corollary \ref{l3new} for ribbons uniquely determined by $\alpha$ and $p$.
\begin{theorem}\label{corl2} Let  $T\in Tab(\nu,\alpha)$  and  $p$ an overlapping partition for $\alpha$. Then,

$(a)$ $T \in\rm LR_{\nu, R^p_\alpha}$ only if $\#\mathcal{D}(\widehat T)\ge p_1$,

$(b)$ if $T$ has no negligible critical numbers and $\mathcal{C}(T)\neq\emptyset$, $T \in\rm LR_{\nu, R^p_\alpha}$ if and only if
every critical number of $T$ is  $p$-ineffective.
\end{theorem}
\begin{proof}$(a)$ The number of columns of length two of $R_\alpha^p$ is $p_1$. Since $T$ has no negligible critical numbers, to avoid columns of length two filled with the same letter, we need that the descent set of $\widehat T$ has at least $p_1$ elements.

$(b)$ It is the translation of  Proposition \ref{l2} according  to the Definition \ref{p}.
\end{proof}

\section{ Characterization of   monotone ribbon  LR coefficients positivity  by means of linear inequalities}
\label{sec:positivity}
Throughout this section we consider  $\alpha$ a partition  with parts of length at least $2$, and overlapping partition $p$.
 Theorem \ref{corl2} says that
 $c_{R_{\alpha}^p}^{\nu}>0$ if and only if, whenever there exists $T\in Tab(\nu,\alpha)$ without negligible critical numbers and  $\mathcal{C}(T)\neq\emptyset$, then
every critical number of $T$ is  $p$-ineffective.
Theorem \ref{ThmNec}  gives a set of linear inequalities on the triple of partitions  $(\alpha,p,\nu)$ as necessary and sufficient conditions for the positivity of $c_{R^p_{\alpha}}^\nu$.
We split the proof of the {\em only if}  and {\em if} parts of Theorem \ref{ThmNec} into two subsections respectively.

 \subsection{Proof of the {\em only if part} of Theorem \ref{ThmNec} }
 If $c_{R_{\alpha}^p}^{\nu}=\#\rm LR_{\nu,R^p_\alpha}>0$ then  there exists $T\in\rm LR_{\nu,R^p_\alpha}\subseteq Tab(\nu,\alpha)$ and $\alpha\preceq \nu$. Let $p=(p_1,\dots,p_{\ell(\alpha)-1},0)$ where $\{p_1,\dots,p_{\ell(\alpha)-1}\}$ is a multiset of $\{\ell(\alpha)-k,\dots,1\}$  such that $p_{\ell(\tilde \alpha_i)}=p_{\ell(\tilde \alpha_i)+1}$ and $\alpha=\tilde \alpha_1\cdots\tilde\alpha_k$ with $R_{\tilde\alpha_i}$, $1\le i\le k$, the connected components of $R_\alpha$.
Therefore $T\in Tab(\nu,\alpha)$ with $\mathcal{D}(\widehat T)=\{\sum_{j=1}^s\alpha_j: s\in S\}\subseteq \mathcal{S}(\alpha)$ for some subset $S=\{s_1<\cdots<s_{|S|}\}\subseteq\{1,\dots,\ell(\alpha)-1\}$ satisfying \begin{equation}\label{eq:overlapdesc}[\{1,\dots,\ell(\alpha)-1\}\setminus\{\ell(\tilde\alpha_1),\dots,\ell(\tilde\alpha_1\tilde\alpha_2\cdots\tilde\alpha_{k-1})\}]\subseteq S=\{s_1<\cdots<s_{|S|}\}\subseteq\{1,\dots,\ell(\alpha)-1\}.\end{equation}
Observe that $|\{\sum_{j=1}^s\alpha_j: s\in \{s_i,\dots,s_{|S|}\}\}|=|\{s_i,\dots,s_{|S|}\}|\ge p_{s_i}$, for $1\le i\le |S|\le \ell(\alpha)-1$. Because $\alpha\preceq \nu$, by Remark \ref{re:ineqdom},
$\nu_i\le \alpha_i+\cdots+\alpha_{\ell(\alpha)}$, for $i\in\{1,\dots,\ell(\nu)\}$.
On the other hand, the $\alpha_i$ $i$'s constitute the $i$-th horizontal strip $\nu^i/\nu^{i-1}$ of $T$ whose rows belong to the first $\min\{i,\ell(\nu)\}$ rows of $T$, for $i\in\{1,\dots,\ell(\alpha)\}$.
Consider the SYT $\widehat T$ and $i\in\{1,\dots,\ell(\nu)\}$. If  $\alpha_1+\cdots+\alpha_i+\cdots+\alpha_s$, $i\le s\in S$, is a descent of $\widehat T$  in the $i$th row of $\widehat T$,  then $\alpha_1+\cdots+\alpha_s+1$ belongs to a row of $\widehat T$ strictly below row $i$. That is, for $1\le i\le q\le\ell(\alpha)-1$, if $\alpha_1+\cdots+\alpha_i+\cdots+\alpha_q$ is a descent of $\hat T$, then either $\alpha_1+\cdots+\alpha_q$ belongs to a row of $\widehat T$ strictly above row $i$, or $\alpha_1+\cdots+\alpha_q+1$ belongs to a row of $\widehat T$ strictly below row $i$. Observe that $|S\cap \{i,\dots,\ell(\alpha)-1\}|$ is the maximum number of descents of  $\hat T$ in  row $i$, and,  simultaneously, is at least equal to the  overlapping number $p_i$, the number of columns of length two among the last $\ell(\alpha)-i+1$ rows of $R_\alpha$,
\begin{equation} |S\cap \{i,\dots,\ell(\alpha)-1\}|\ge p_i.\end{equation} Hence $\nu_i\le \alpha_i+\cdots+\alpha_{\ell(\alpha)}-|S\cap \{i,\dots,\ell(\alpha)-1\}|\le \alpha_i+\cdots+\alpha_{\ell(\alpha)}-p_i$, for $i\in\{1,\dots,\ell(\nu)\}$.
   $\Box$

\subsection{Proof of the {\em if part} of Theorem \ref{ThmNec}}
\label{sec:rotation}
 Given 
the triple of partitions, $\nu$, and
$\alpha$, with parts $\geq 2$, and overlapping partition $p$, satisfying the linear inequalities on the  right hand side of \eqref{ineq},
the goal is now to exhibit  a  SSYT $T\in\rm LR_{R^p_\alpha,\nu}$.
In other words, assuming the linear inequalities on the  right hand side of  \eqref{ineq}, we  construct a SSYT $T\in Tab(\nu,\alpha)$ without  negligible critical numbers and $p$-effective critical numbers.  In more detail, we pick $T\in Tab(\nu,\alpha)$ with the {\em canonical filling}, thus  without negligible critical numbers,  and then, if it has $p$-effective critical numbers, one  modifies its filling according to a certain {\em rotation} procedure to remove them so that the new tableau  is  in $\rm LR_{R^p_\alpha,\nu}$. The application of rotation procedure does not create negligible critical numbers.
The linear inequalities on the right hand side of \eqref{ineq} guarantee that our rotation procedure is successful.
 \begin{remark} Let $\alpha\preceq \nu$ and $p$ an overlapping partition for $\alpha$.

$(a)$ If $\ell(\nu)=\ell(\alpha)$, given $T\in Tab(\nu,\alpha)$, the first entry of each row $i$ of $T$ is $i$ and $T$ has no critical numbers of any kind. The descent set of $\widehat T$ is $S(\alpha)$ and
every $T\in Tab(\nu,\alpha)$ is a companion tableau for an LR filling of $R^p_\alpha$. In this case, the linear inequalities \eqref{ineq} are trivially satisfied because below each row $i$ of $T$ one has at least $\ell(\alpha)-i\ge p_i$ entries and thereby $\nu_i\le\alpha_i+\cdots+\alpha_{\ell(\alpha)}-\ell(\alpha)+i\le\alpha_i+\cdots+\alpha_{\ell(\alpha)}-p_i$.
Also $c_{R_\alpha^p}^\nu=K_{\nu,\alpha}$.

$(b)$ If $\ell(\nu)=1$ then $\nu=(|\alpha|)$, $p=0$,  and $|Tab(\nu,\alpha)|=1$. The descent set of the sole $\widehat T$ is $S(\alpha)=\emptyset$, and linear inequalities \eqref{ineq} are trivially  satisfied with $p=0$. One has $c_{R_\alpha^p}^\nu=K_{\nu,\alpha}=1$.
\end{remark}
We shall consider $\nu$  with at least two rows and less than $\ell(\alpha)$ rows, $2\le \ell(\nu)
<\ell(\alpha)$.

We start with the case $\ell(\nu)=\ell(\alpha)-1$.

\begin{lemma}\label{p1} Let $\nu\in[\alpha,(|\alpha|-p_1,p_1)]$  with $\ell(\nu)=\ell(\alpha)-1$,
such that
 \begin{equation}\nu_i\leq \sum_{j\geq i}\alpha_j-p_i,\,\text{ for }1\leq i< \ell(\alpha).\nonumber\end{equation}
 Then, $c_{R_{\alpha}^p}^{\nu}>0$.
\end{lemma}
\begin{proof} Let $s:=\ell(\alpha)$.
Let $T\in Tab(\nu,\alpha)$, with the {\em canonical filling}, and note that since $\ell(\nu)=s-1$, then the first column of $T$ has all letters of $[s]\setminus\{j\}$, for some $2\leq j\leq s$, and necessarily row $ j-1$  contains $\alpha_j$ letters $j$.  That is, the first entry of row $i$ of $T$ is $i$, for $i=1,\dots,j-1$, and is $i+1$ for $i=j,\dots,s-1$. Thus, $\chi^k_k\neq 0$, $1\le k\le j-1$, $\chi^{j-1}_j=\alpha_j$, $\chi^k_{k+1}\neq 0$, $j\le k\le s-1$, and $\chi_j^k=0$, $k\ge j$. The only row of $T$ which can potentially be  critical   is row $j-1$, since by Proposition \ref{critic}, $\chi_{j-1}^{j-1}\neq 0$ and $\chi_j^{j-1}=\alpha_j$.  That is, the rows $j-1$ and $j$ of $\varphi_{R_\alpha}(T)$ look like
\begin{equation}\label{critic1}
\begin{tabular}{cll|l|lllll}
\cline{4-9}
$<\alpha_{j-1}>$  &                       &  & $x$ & \multicolumn{1}{l|}{$\cdots$} & \multicolumn{1}{l|}{$x$} & \multicolumn{1}{l|}{$j-1$} & \multicolumn{1}{l|}{$\cdots$} & \multicolumn{1}{l|}{$j-1$} \\ \cline{2-9}
\multicolumn{1}{c|}{$<\alpha_{j}>$} & \multicolumn{1}{l|}{$j-1$} & $\cdots$ & $j-1$       &                 &                       &                       &                       &                       \\ \cline{2-4}
\end{tabular}
\end{equation}
or
\begin{equation}\label{critic2}
\begin{tabular}{cll|l|ll}
\cline{4-6}
$<\alpha_{j-1}>$                      &                       &  &$j-1$  & \multicolumn{1}{l|}{$\cdots$} & \multicolumn{1}{l|}{$j-1$} \\ \cline{2-6}
\multicolumn{1}{c|}{$<\alpha_{j}>$} & \multicolumn{1}{l|}{$j-1$} & $\cdots$ & $j-1$ &                       &                       \\ \cline{2-4}
\end{tabular}
\end{equation}
where the word $x\cdots x=(j-2)^r$, $r\ge 0$,  may be empty.
If $j$ is not a critical number, or if it is a $p$-ineffective critical number \eqref{critic1} then, by Proposition \ref{l2}, $\varphi_{R_\alpha}(T)$ is a  skew SSYT.

Assume now that $j$ is $p$-effective critical  number of $T$ \eqref{critic2}. In particular, this means that $\chi_{j-1}^{j-1}=\alpha_{j-1}$. Notice that if $j=s$, then
$$\nu_{j-1}=\nu_{s-1}=\alpha_{s-1}+\alpha_s\leq\alpha_{s-1}+\alpha_s-p_{s-1},$$
which implies $p_{s-1}=0$. That is, rows $j-1$ and $j$ of $R_\alpha^p$ do not overlap, which contradicts the fact that $j=s$ is $p$-effective critical.
So, we must have $2\leq j<s$, and, in particular, row $j$ of $T$ has at least one integer  $j+1$. Table \ref{table:pcr} depicts rows  $j-1$ and $j$ of $T$, where $\ast$ denotes   $\chi_{j+1}^{j-1}> 0$ boxes with the letter $j+1$, or the empty cell  if $\chi_{j+1}^{j-1}=0$,

\begin{table}[H]
\centering

\begin{tabular}{r|c|c|c|c|ccccccc}
\cline{2-12}
row $j-1$ &$j-1$ & $j-1$&$\cdots$& $j-1$ & \multicolumn{1}{l|}{$\cdots$} & \multicolumn{1}{l|}{$j-1$} & \multicolumn{1}{l|}{$\bf{j}$} & \multicolumn{1}{l|}{$\cdots$} & \multicolumn{1}{l|}{$\bf{j}$} & \multicolumn{1}{l|}{$\bf{j}$} & \multicolumn{1}{l|}{$\bf{\ast}$} \\ \cline{2-12}
 row $j$ & $\bf{j+1}$ & $\bf{j+1}$& $\cdots$ &$\bf{j+1}$ & \multicolumn{6}{l}{$\cdots$}                                                                                                                          \\ \cline{2-5}
\end{tabular}
\caption{Rows $j-1$ and $j$ of $T$}
\label{table:pcr}
\end{table}

Perform the procedure Rotation described in Table \ref{table:proc1}  with $\ell=j-1$,  $a=j$ and $ b=j+1$ on the tableau $T$.
\begin{table}[H]
\centering

\begin{tabular}{l}
\hline
 {\bf Procedure:} Rotation\\ \hline
 {\bf Data:} Tableau $T$; Integers $a$ and  $\ell$;\\
 {\bf Begin}\\
 \quad Let $\ell'>\ell$ be  the smallest integer such that row $\ell'$ of $T$ has an\\
 \quad integer $b$ greater or equal to the rightmost letter in row $\ell$;\\
  \quad Rotate by one turn in anticlockwise order all letters greater or equal to\\
\quad $a$ in row $\ell$, and all letters $b$ of row $\ell'$ of $T$;\\
  {\bf Stop}\\\hline
\end{tabular}
\caption{Procedure: Rotation}
\label{table:proc1}
\end{table}

That is, rotate the highlight letters $j$ and $j+1$ of $T$ (Table \ref{table:pcr}) in anticlockwise order to obtain the rows  shown in Table \ref{table:PCa2}, and denote by $T'$ the tableau obtained from $T$ by this operation.

\begin{table}[H]
\centering

\begin{tabular}{r|c|c|c|c|ccccccc}
\cline{2-12}
row $j-1$ & $j-1$ &$j-1$&$\cdots$& $j-1$ & \multicolumn{1}{l|}{$\cdots$} & \multicolumn{1}{l|}{$j-1$} & \multicolumn{1}{l|}{$\bf{j}$} & \multicolumn{1}{l|}{$\cdots$} & \multicolumn{1}{l|}{$\bf{j}$} & \multicolumn{1}{l|}{$\bf{j+1}$} & \multicolumn{1}{l|}{$\bf{\ast}$} \\ \cline{2-12}
 row $j$ & $\bf{j}$ & $\bf{j+1}$& $\cdots$&$\bf{j+1}$ & \multicolumn{6}{l}{$\cdots$}                                                                                                                          \\ \cline{2-5}
\end{tabular}
\medskip
\caption{Rows $j-1$ and $j$ of $T'$}
\label{table:PCa2}
\end{table}

 {\em We recall that we are assuming $\alpha $ a partition and thus $\alpha_{j-1}\ge \alpha_j\ge\alpha_{j+1}$}. The new tableau $T'$ is still semistandard and the integer $j$ is no longer critical, since $\chi_j^{j}>0$.
Notice, however, that if $\chi_{j+1}^j=1$ in $T$, then in $T'$ the integer $j+1$ is  critical.
Therefore, if $\chi_{j+1}^j>1$ in $T$, or $\chi_{j+1}^j=1$ in $T$ and $j+1$ is  $p$-ineffective critical, then $\varphi_{R_\alpha^p}(T')$ is  a skew SSYT.
So, assume   $\chi_{j+1}^j=1$, $j$  $p$-effective critical in $T$, and  in addition rows $j$ and $j+1$ of $R^p_\alpha$ overlap ($j+1$ is $p$-effective in $T'$).
If $j+1=s$, (Table \ref{table:PC11}) then $\nu_{s-1}=1$ and $p_{s-2}=2$,
\begin{table}[H]
\centering

\begin{tabular}{r|c|c|c|c|ccccc}
\cline{2-10}
row $s-2$ &$s-2$ &  $\cdots$ & \multicolumn{1}{l|}{$s-2$} & \multicolumn{1}{l|}{$\bf{ s-1}$} & \multicolumn{1}{l|}{$\cdots$} & \multicolumn{1}{l|}{$\bf{s-1}$} & \multicolumn{1}{l|}{$\bf{s}$} & \multicolumn{1}{l|}{$\cdots$} & \multicolumn{1}{l|}{$\bf{s}$} \\ \cline{2-10}
 row $s-1$ & $\bf{s-1}$                                                                                                                          \\ \cline{2-2}
\end{tabular}
\caption{Rows $s-2$ and $s-1$ of $T'$}
\label{table:PC11}
\end{table}
\noindent and
$\nu_{s-2}=\alpha_{s-2}+ \alpha_{s-1}+(\alpha_s-1)\leq \alpha_{s-2}+\alpha_{s-1}+\alpha_s-p_{s-2},$
that is,
$p_{s-2}\leq 1.$
A contradiction, then the rows $s-1$ and $s$ of $R^p_\alpha$ cannot overlap, and $j+1=s$ is not $p$-effective critical in $T'$.

So we must have $j+1<s$, and  there must be integers other than $j+1$ in row $j$ of $T'$, since otherwise the rows of $T'$ below row $j$ would have only one box, which in turn would imply $2\le\alpha_{j+2}=1$, a contradiction. So there are letters $j+2$ in row $j$ and the number of letters $j+2$ below row $j-1$ is $\alpha_{j+2}\ge 2$ (Table \ref{table:PC2}). Apply the procedure Rotation with  $a=j+1$ and $\ell=j-1$ to $T'$,
\begin{table}[H]
\centering
\begin{tabular}{r|c|c|c|c|ccccccccccc}
\cline{2-13}
row $j-1$ & $j-1$ &$j-1$&$\cdots$& $j-1$ & \multicolumn{1}{l|}{$\cdots$} & \multicolumn{1}{l|}{$j-1$} & \multicolumn{1}{l|}{$j$} & \multicolumn{1}{l|}{$\cdots$} & \multicolumn{1}{l|}{$j$} & \multicolumn{1}{l|}{$\bf{j+1}$} & \multicolumn{1}{l|}{$\cdots$}& \multicolumn{1}{l|}{$\bf{j+1}$} \\ \cline{2-13}
 row $j$ & $j$ & $\bf{j+2}$& $\cdots$&$\bf{j+2}$ & \multicolumn{6}{l}{$\cdots$}                                                                                                                          \\ \cline{2-5}
\end{tabular}
\medskip
\caption{Rows $j-1$ and $j$ of $T'$}
\label{table:PC2}
\end{table}
\noindent and let $T''$ be the resulting tableau (Table \ref{table:PC3}),
\begin{table}[H]
\centering
\begin{tabular}{r|c|c|c|c|ccccccccccc}
\cline{2-13}
row $j-1$ & $j-1$ &$j-1$&$\cdots$& $j-1$ & \multicolumn{1}{l|}{$\cdots$} & \multicolumn{1}{l|}{$j-1$} & \multicolumn{1}{l|}{$j$} & \multicolumn{1}{l|}{$\cdots$} & \multicolumn{1}{l|}{$j$} & \multicolumn{1}{l|}{$\bf{j+1}$} & \multicolumn{1}{l|}{$\cdots$}& \multicolumn{1}{l|}{$\bf{j+2}$} \\ \cline{2-13}
 row $j$ & $j$ & $\bf{j+1}$& $\cdots$&$\bf{j+2}$ & \multicolumn{6}{l}{$\cdots$}                                                                                                                          \\ \cline{2-5}
\end{tabular}
\medskip
\caption{Rows $j-1$ and $j$ of $T''$}
\label{table:PC3}
\end{table}

This new tableau is semistandard and $j+1$ is no longer a critical number, since there is now a letter $j+1$ in row $j$. Also, since $\alpha_{j+2}\geq 2$, there must be integers $j+2$ below row $j-1$. Thus, $T''$ does not have critical numbers and then $\varphi_{R_\alpha^p}(T'')$ is  a skew SSYT.
\end{proof}

\begin{remark}
Notice that when applying the procedures, described in the proof of the result above, to a tableau $T$ with only one critical number $j$ in row $j-1$, we only modify rows $j-1$ and $j$ of $T$. Moreover, in row $j$, only the integers $j+1$, and possible $j+2$, are acted upon.
The rows  above row $j-1$, as well as the letters in row $j-1$ to the left of the letters $j$, are not considered for the application of the procedure.
\end{remark}

\begin{example}
Let $\nu=(8,1)$ and $\alpha=(3,3,3)$, and consider the tableau $$T=\young(11122233,3)\in Tab(\nu,\alpha).$$
The tableau $T$ has only one critical  number: the integer $2$, that is,
 the descent of $\widehat T$ is $\{\alpha_1+\alpha_2\}$. If $R_\alpha=R_{\alpha_1}\oplus R_{(\alpha_2,\alpha_3)}$,
equivalently, $p=(1,1,0)$, then $\varphi_{R_\alpha}(T)$ is SSYT and the integer $2$ is not  $p$-effective critical, and so
$$\varphi_{R_\alpha}(T)=\young(:::::111,::111,112)$$
is a skew SSYT.  Note also, $\nu_1=8\le \sum_{i=1}^3\alpha_i-1,\,\nu_2=1\le \alpha_2+\alpha_3-1,\; \nu_3\le\alpha_3-0$.

If $p=(1,0,0)$ then $R_\alpha=R_{(\alpha_1,\alpha_2)}\oplus R_{(\alpha_3)}$, $2$ is a $p$-effective critical number and $\varphi_{R_\alpha}(T)$ is not SSYT. Perform the procedure Rotation on $T$ with $a=2$ and $\ell=1$ to get
$$T=\young(11122233,3)\rightarrow\young(11122333,2)=T'.$$
The tableau $T'$ has no  effective critical numbers for the overlapping partition $p=(1,0,0)$, the descent set of $\widehat T'$ is $\{\alpha_1\}$, and therefore $$\varphi_{R_\alpha^p}(T')=\young(:::::111,:::112,111)$$
is a skew SSYT.
There is no connected  LR ribbon of shape $R_\alpha$ and content $\nu$: if $p=(2,1,0)$, $\nu_1=8>|\alpha|-2=9-2$.
\end{example}

\begin{example}
Let $\nu=(9,3)$ and $\alpha=(4,3,3,2)$, and consider the tableau $$T=\young(111122233,344)\in Tab(\nu,\alpha).$$
The letter $2$ is the only critical number of $T$, and is effective when we consider the overlapping partition $p=(3,2,1,0)$. So, we apply the procedure Rotation on $T$ with $a=2$ and $\ell=1$:
$$T=\young(111122233,344)\rightarrow \young(111122333,244)=T'.$$
In $T'$, the number $2$ is no longer  critical. However, a new critical number was created: the number $3$. So we apply   Rotation on $T'$ with $a=3$ and $\ell=1$ to get the tableau $$T''=\young(111122334,234),$$
which has no critical numbers. It follows that
$$\varphi_{R_\alpha}(T'')=\young(:::::1111,:::112,:113,14)$$
is a skew SSYT. Note $\nu_1=9\le |\alpha|-3=12-3$, $\nu_2=3\le 8-2$, $\nu_3=0\le 5-1$.
\end{example}

 \begin{lemma}\label{p2} Let $\nu\in[\alpha,(|\alpha|-p_1,p_1)]$
with  $\ell(\nu)=\ell(\alpha)-k$, $1\le k\le \ell(\alpha)-2$, and satisfying
 \begin{equation}\nu_i\leq \sum_{j\geq i}\alpha_j-p_i,\,\text{ for }1\leq i\leq \ell(\nu).\nonumber \end{equation}
If $T$ is the SSYT with  canonical filling in $ Tab(\nu,\alpha)$ and  has $\mathcal{C}(T)=\{j_1,j_2,\ldots,j_k\}$ with $j_{i+1}=j_i+1$, for $i=1,\ldots,k-1$,
 then, $c_{R_{\alpha}^p}^{\nu}>0$.
\end{lemma}
\begin{proof}
 Let $T$ be the canonical filling in $Tab(\nu,\alpha)$ with $\mathcal{C}(T)=\{j_1,j_2,\ldots,j_k\}$ such that $j_{i+1}=j_i+1$ for $i=1,\ldots,k-1$.
Then,  the first column of $T$ has all letters of $[s]\setminus\{j_1,j_2,\ldots,j_k\}$, and row $j_1-1$ has $\alpha_i$ letters $j_i$, for $i=1,\ldots,k$. We are assuming that $j_1$ is  critical but $j_k+1$ is not, row $j_1-1$ also has $\alpha_{j_1-1}$ letters $j_1-1$ and $0\leq\chi_{j_k+1}^{j_1-1}<\alpha_{j_k+1}$ letters $j_k+1$,
thus, row $j_1-1$ of $T$ satisfy
 $$\nu_{j_1-1}=\alpha_{j_1-1}+\alpha_{j_1}+\cdots+\alpha_{j_k}+\left(\alpha_{j_k+1}-\chi_{j_k+1}^{j_1}\right)\leq \alpha_{j_1-1}+\alpha_{j_1}+\cdots+\alpha_{s}-p_{j_1-1},$$
that is,
\begin{equation}\label{eq1}
p_{j_1-1}\leq \alpha_{j_k+2}+\cdots+\alpha_s+\chi_{j_k+1}^{j_1}
\end{equation}
where $0<\chi_{j_k+1}^{j_1}$.
The number  of $p$-effective critical numbers of $T$, which  are at most $k$, must be less than or equal to $p_{j_1-1}$. Thus, by \eqref{eq1}, there are at least $p_{j_1-1}$ integers greater than or equal to $j_k+1$ below row $j_1-1$ of $T$ and we can perform  procedure Rotation 1 on $T$ with $\mathcal{C}(T)=\{j_1,\ldots,j_k\}$ and $\overline{\ell}=j_1-1$.

\begin{table}[H]
\centering

\begin{tabular}{l}
\hline
 {\bf Procedure:} Rotation 1\\ \hline
 {\bf Data:} Tableau $T$; Set $\mathcal{C}(T)=\{j_1<\ldots<j_k\}$; Integer $\overline{\ell}$;\\
 {\bf Begin}\\
 \quad {\bf For} $i=1$ to $k$ do\\
 \quad\quad {\bf If} $j_i$ is an  $p$-effective critical point of $T$, perform procedure\\
 \quad\quad Rotation (Table \ref{table:proc1}) with $a=j_i$ and $\ell=\overline{\ell}$;\\
 \quad \quad {\bf End If}\\
 \quad{\bf End For}\\
  {\bf Stop}\\\hline
\end{tabular}
\caption{Procedure: Rotation 1}
\label{table:proc2}
\end{table}

Let $T'$ be the tableau resulting from the application of Procedure  Rotation 1 (Table \ref{table:proc2}) on $T$. Notice that the assumption of $\alpha$ a partition and the canonical filling of $T$ asserts that $T'$ is semistandard. Moreover, the integers $j_1,\ldots,j_k$ are not critical  numbers of $T$ since there are letters $j_1,\ldots,j_k$ below row $j_1-1$. However, the operations performed on $T$ to produce $T'$ may create new critical  numbers, all of which are in row $j_1-1$. This only happens  when all letters of an integer, say $r>j_k$, are sent to row $j_1-1$. Note that $r$ must be one of the first $k$ letters below row $j_1-1$ which are greater or equal to the rightmost letter of row $j_1-1$. Let $r_1,\ldots,r_{k'}$ be the new critical  numbers created in $T'$. If they are  $p$-effective, then by \eqref{eq1}, we must have
$$k+k'\leq p_{j_1-1}.$$
This means that below row $j_1-1$ of $T'$ there exist at least $k'$ integers greater or equal to the rightmost letter of row $j_1-1$, and we can perform procedure Rotation 1 on $T'$ with $\mathcal{C}(T)=\{r_1,\ldots,r_{k'}\}$ and $\overline{\ell}=j_1-1$, obtaining a new tableau $T''$, where $r_1,\ldots,r_{k'}$ are not critical {\olga }. Again, new critical  numbers $q_1,\ldots,q_{k''}$, with $r_{k'}<q_1,\ldots,q_{k''}$ may occur, in which case we repeat the process. Note that since the number of  $p$-effective critical numbers cannot exceed $p_{j_1-1}$, this process must terminate.

Therefore, the tableau $\widetilde{T}$ obtained after this procedure is semistandard and has no critical  numbers. We can conclude that $\varphi_{R_\alpha}(\widetilde{T})$ is semistandard.
\end{proof}

\begin{remark}Notice that Lemma \ref{p1} is a special case of Lemma \ref{p2}.
Also, notice that the tableau $\widetilde{T}$ obtained after the process described in the result above only differs from $T$ between the rows $j_1-1$, the ones having the critical  numbers, and some row below it, say $j$, from the leftmost integer of $j$ until the last integer in row $j$ that has been rotated to row $j_1-1$.
\end{remark}

\begin{example}
Let $\nu=(9,2,2,2)$, $\alpha=(3,2,2,2,2,2,2)$, and consider the overlapping vector $p=(6,5,4,3,2,1,0)$ and the tableau
$$T=\young(111223344,55,66,77)\in Tab(\nu,\alpha).$$
The letters $2,3$ and $4$ are consecutive  $p$-effective critical numbers of  $T$. Apply procedure Rotation 1 with $\mathcal{C}(T)=\{2,3,4\}$ and $\overline{\ell}=1$:
$$T\rightarrow\young(111233445,25,66,77)\rightarrow\young(111234455,23,66,77)\rightarrow\young(111234556,23,46,77)=T'.$$

Now, the letter $5$ is the only critical  number  of the  resulting tableau $T'$. So, we apply Rotation 1 again on $T'$ with $\mathcal{C}(T)=\{5\}$ and   $\overline{\ell}=1$:
$$T'=\young(111234556,23,46,77)\rightarrow \young(111234566,23,45,77)=T''.$$
Now, the letter $6$ is the only critical  number of the  resulting tableau $T''$. So, we apply Rotation 1 again on $T'$ with $\mathcal{C}(T)=\{6\}$ and   $\overline{\ell}=1$:
$$T''=\young(111234566,23,45,77)\rightarrow \young(111234567,23,45,67)=\widetilde{T}.$$
The tableau $\widetilde{T}$ has no critical  numbers and thus $\varphi_{R_\alpha}(\widetilde{T})$ is a skew SSYT.
\end{example}

We  now can prove the general case.

\begin{theorem}\label{ThmSuf}
Let $\nu\in[\alpha,(|\alpha|-p_1,p_1)]$   where  $\ell(\nu)=\ell(\alpha)-k$, $1\le k\le \ell(\alpha)-2$,
and satisfying \begin{equation}\nu_i\leq \sum_{j\geq i}\alpha_j-p_i,\,\text{ for }1\leq i\leq \ell(\alpha).\nonumber\end{equation}
Then, $c_{R_{\alpha}^p}^{\nu}>0$.
\end{theorem}
\begin{proof}
Let $T\in Tab(\nu,\alpha)$ with the canonical filling, and  $\mathcal{C}(T)=\{j_1,\ldots,j_k\}$. Write
$$\mathcal{C}(T)=A_1\cup A_2\cup\cdots\cup A_r$$
the  set partition of  $\mathcal{C}(T)$ such that in each set $A_i$ all critical  numbers are consecutive, and if $a\in A_i$ and $b\in A_{i+q}$, for some $q>0$, then $a<b$ and $b-a\geq 2$.

Notice that in this case, the $\alpha_a$ letters $a$ must be all in some row $\ell$,  and the $\alpha_b$ letters $b$ must be all in some row $\ell'$ of $T$, with $\ell<\ell'$.

Apply the procedure described in  Lemma \ref{p2}  to the set of consecutive critical  numbers in $A_1$. This procedure may use some integers from $A_2\cup\cdots\cup A_r$ in its Rotation routines. If this is the case, then in the resulting tableau $T'$, some of the critical  numbers in $A_2\cup\cdots\cup A_r$ may no longer be critical numbers, since some of them may have been brought, by rotation, to a higher row of the tableau. Nevertheless, no new critical numbers are created by this process. So, in $T'$, the critical  numbers
can be partitioned as
$$A'_2\cup\cdots\cup A'_r,$$
where $A'_i\subseteq A_i$ for all $i=2,\ldots,r$.

Repeating the process, until no more critical points remain, we obtain a tableau $\widetilde{T}$ such that $\varphi_{R_\alpha}(\widetilde{T})$ is a skew SSYT.
\end{proof}

\begin{example}
Let $\nu=(13,13,2)$, $\alpha=(4,3^7)$ and $p=(8,7,6,5,4,3,2,1,0)$. The tableau
$$T=\young(1111222333444,5556667778889,99)\in Tab(\nu,\alpha)$$
has the critical points $2,3,4,6,7,8$, which can be partitioned as
$$A_1=\{2,3,4\}\cup A_2=\{6,7,8\},$$
according to the proof of the theorem above. We start by removing the critical  numbers in $A_1$:
$$T\rightarrow\young(1111223334445,2556667778889,99)\rightarrow\young(1111223344455,2356667778889,99)$$

$$\rightarrow\young(1111223344555,2346667778889,99)\rightarrow\young(1111223344556,2345667778889,99)=T'.$$

After the application of the procedure described in Lemma \ref{p2} to the critical  numbers in $A_1$, we get the tableau $T'$, whose only critical olga number is
$$\{8\}=A'_2\subset A_2.$$
So, we apply the procedure described in Lemma \ref{p2} again to the critical number in $A'_2$:
$$T'\rightarrow\young(1111223344556,2345667778899,89)=\widetilde{T}.$$
The resulting tableau $\widetilde{T}$ has no critical numbers and thus, the skew tableau
$$\varphi_{R_\alpha}(\widetilde{T})=\young(::::::::::::::::1111,::::::::::::::112,::::::::::::112,::::::::::112,::::::::112,::::::122,::::222,::223,223)$$
is a skew SSYT.
\end{example}

\section{Classification of monotone ribbons with full Schur support}
\label{sec:full}
Theorem \ref{ThmNec}, characterizing the positivity of  monotone ribbon LR coefficients, $c^\nu_{R^p_\alpha}>0$, by means of linear inequalities, may be rephrased in the language of the Schur support of $R^p_\alpha$.
Let $\nu\in[\alpha,(|\alpha|-p_1,p_1)]$, $\alpha$ a partition with parts $\ge 2$. Then
 \begin{equation}\text{$\nu\in[R_{\alpha}^p]$ if and only if  $\label{equ2}\nu_i\leq \sum_{q=i}^{\ell(\alpha)}\alpha_q-p_i$,\;\;$1\leq i\leq \ell(\alpha)$.}\end{equation}
By Remark \ref{re:ineqdom}, if $\alpha\preceq \nu$ one has
 $\displaystyle \nu_i\leq \sum_{q= i}^{\ell(\alpha)}\alpha_q$, for $1\le i\le \ell(\alpha)$. Hence, if $\nu\in[\alpha,(|\alpha|-p_1,p_1)]$ then $\nu_1\leq \sum_{q=1}^{\ell(\alpha)}\alpha_q-p_1$, and
 because one  has $p_{i}=0$, for $ \ell(p)<i\le\ell(\alpha)$,   the  inequalities \eqref{equ2}  are always satisfied for  $\ell(p)<i\le\ell(\alpha)$.
 Note that, when $\ell(p)\ge 2$,    $p_{i+1}-1\ge 0$,  $ i\in\{1,\dots, \ell(p)-1\}$.
Recall Definition \ref{witness} and
$\displaystyle \varrho_i=\sum_{q=i+1}^{\ell(\alpha)}\alpha_q-p_{i+1}+1>0$,   where $\varrho_i-1$ is the total number of columns in the last $\ell(\alpha)-i$ rows of $R^p_\alpha$, for $1\le i\le \ell(p)-1$.
\begin{remark} \label{re:varrho}Because the parts of  $\alpha$ are $\ge 2$, and $p_i=p_{i+1}$ or $p_i=p_{i+1}+1$,  $|\alpha|-p_1>\varrho_1> \cdots> \varrho_{\ell(p)-1}$.
\end{remark}
Thus, the negation of \eqref{equ2}  characterizes  the partitions in the interval $[\alpha,(|\alpha|-p_1,p_1)]$ which are not in the support of $R_{\alpha}^p$.
\begin{corollary}\label{char}
 Let $\nu\in[\alpha,(|\alpha|-p_1,p_1)]$ and  $\alpha$  a partition with parts $\ge 2$. Then,
 if   $\ell(p)=0,1$, $ [R_{\alpha}^p]=[\alpha,(|\alpha|-p_1,p_1)]$, and, if $\ell(p)\ge 2$, the following are equivalent

$(a)$  $\nu\notin [R_{\alpha}^p]$ if and only if there exists $ i\in\{1,\dots, \ell(p)-1\}$ such that
\begin{equation*}\nu_{i+1}\geq \sum_{q\geq i+1}\alpha_q-p_{i+1}+1\Leftrightarrow \nu_{i+1}\ge \varrho_i.\end{equation*}

 $(b)$ $\nu\notin [R_{\alpha}^p]$ if and only if, for some $i\in\{1,\dots, \ell(p)-1\}$, $\nu_{i+1}$ exceeds the number of columns in the last $\ell(\alpha)-i$ rows of $R_\alpha^p$.

 $(c)$ \cite[Lemma 4.8]{acm17} $\nu\notin [R_{\alpha}^p]$ if and only if, there exists $i\in\{1,\dots, \ell(p)-1\}$ such that after placing  $\alpha_j$ $j$'s, in row $j$ of $R_\alpha^p$, for $j=1,\dots,i$,  there is no space to place $\nu_{i+1}$ $i+1$'s in the remain $\ell(\alpha)-i$ rows of $R_\alpha^p$ without avoiding the violation of the column standard condition of the filling.

  $(d)$ $\nu\notin [R_{\alpha}^p]$, if, for every $T\in Tab(\nu,\alpha)$, there exists $i\ge 1$ such that $|\mathcal{D}(\widehat T)\cap\{\sum_{q\ge 1}^{j}\alpha_q: i+1\le j\le \ell(\alpha)\}|<p_{i+1}$.
\end{corollary}

\begin{example}
Consider the partition $\alpha=(7,6,6,2,2,2,2)$ with the overlapping partition $p=(6,5,4,3,2,1,0)$. The partition $\nu=(8,7,6,6)$ is in the Schur interval $[\alpha,(27-6,6)]$ of $R_{\alpha}^p$, but not in its support since $\nu_4=6\geq \varrho_3=\alpha_4+\alpha_5+\alpha_6+\alpha_7-p_4+1=2+2+2+2-3+1=6$. Therefore, $[R_{\alpha}^p]\subsetneqq [\alpha,(27-6,6)].$
\end{example}

   Theorem \ref{Tchar} characterizes  the monotone ribbons $R_\alpha^p$ with full Schur support in terms of their partition skew shape $\alpha$ and  the overlapping partition $p$. In Definition \ref{witness}  a sequence of $\ell(p)-1$ witness vectors  $\tilde g^i=\{\tilde g^i_j\}_{j=1}^i=\{\left[ \varrho_i-\alpha_j\right]_+\}_{j=1}^i$  with its slack $p_{i+1}-1$, $1\le i\le \ell(p)-1$, is introduced to test the fullness of the Schur support  of $R^p_\alpha$.
Theorem \ref{Tchar} says that if, for some $1\le i\le \ell(p)-1$, the size of the  witness vector $\tilde g^i$    fits the slack $p_{i+1}-1$, that is $\sum_{j=1}^i\tilde g_j^i\le p_{i+1}-1$, then  $R_\alpha^p$ has not  full Schur support. In this case the vector $\tilde g^i$ witnesses that the Schur support $R_\alpha^p$  is not full in the sense that it can be used to exhibit a partition in the Schur interval that is not in $[R_\alpha^p]$. More precisely, $\displaystyle(\alpha_1 +\tilde g_1^i,\dots,
\alpha_i +\tilde g_i^i,\varrho_i,$ $p_{i+1}-1-|\tilde g^i| )^+$, with $\varrho_i-1$ the total number of columns in the last $\ell(\alpha)-i$ rows of $R^p_\alpha$,  is a partition of $|\alpha|$ in the Schur interval  of $R_\alpha^p$ but  not in the support of $R_\alpha^p$.

\subsection{Proof of Theorem \ref{Tchar} }
The {\em ``only if"} part.  Let $\nu \in[\alpha,(|\alpha|-p_1,p_1)]$
such that $\nu\notin[R_{\alpha}^p]$. Then, on one hand, since $\alpha\preceq \nu$, $\sum_{q=1}^{k}(\nu_q-\alpha_q)\ge 0$, $k=1,\dots, \ell(\alpha)$, and on the other hand,  since $\nu\notin[R_{\alpha}^p]$, by Corollary \ref{char}, $\ell(\alpha)\ge 3$, $\ell(p)\ge 2$, and there exists   $1\le i\le \ell(\alpha)-2$  with $p_{i+1}\ge 1$ such that
\begin{equation*}\nu_{i+1}\geq \varrho_i=\sum_{q\geq i+1}\alpha_q-p_{i+1}+1.\end{equation*}
We want to show that the $i$-witness vector $\widetilde{g}^i=(\widetilde{g}^i_1,\ldots,\widetilde{g}^i_i)$ of $R^p_\alpha$ fits its slack $p_{i+1}-1$.
It follows that  $0\le \sum_{q=1}^{i}(\nu_q-\alpha_q)\leq p_{i+1}-1$,  otherwise, we would have
$$\sum_{q=1}^{i+1}\nu_q=\sum_{q=1}^{i}\nu_q+\nu_{i+1}>\sum_{q=1}^{i}\alpha_q+p_{i+1}-1+
\sum_{q=i+1}^{\ell(\alpha)}\alpha_q-p_{i+1}+1=\sum_{q=1}^{\ell(\alpha)}\alpha_q,$$
contradicting the equality $|\alpha|=|\nu|$.

Let $U:=\{j\in\{2,\dots,i\}:\nu_j-\alpha_j<0\}$ (indeed $\nu_1\ge \alpha_1$) and $u:=\max U$. Put $u:=0$ if $U=\emptyset$.

{\em Claim}: There exist $\mu_j\ge \alpha_j$, $j=1,\dots, u$, such that
 \begin{equation}\mu_1\ge\dots \ge \mu_{u-1}\ge \alpha_{u-1}\ge \mu_u=\alpha_{u}>\nu_u\ge \nu_{u+1}\ge\cdots\ge \nu_i\ge \nu_{i+1},\;\text{and}
\end{equation}
 \begin{equation}\label{pp}\sum_{j=1}^u (\mu_j-\alpha_{j})= \sum_{j=1}^{u}(\nu_j-\alpha_{j})\ge 0.\end{equation}
In these conditions, defining $g_j:=\mu_j-\alpha_j\ge 0$, $j=1,\dots,u$, and $g_j:=\nu_j-\alpha_j\ge 0$, $j=u+1,\dots,i$, one has
$\sum_{j=1}^i g_j=\sum_{j=1}^{i}(\nu_j-\alpha_{j})\le p_{i+1}-1$,
$$\alpha_j+g_j=\mu_j\ge \alpha_u>\nu_u\ge \nu_{i+1}\ge\sum_{q\ge i+1} \alpha_q-p_{i+1}+1=\varrho_i, \;j=1,\dots,u,$$
and
$$\alpha_j+g_j=\nu_j\ge \nu_i\ge \nu_{i+1}\ge\sum_{q\ge i+1} \alpha_q-p_{i+1}+1=\varrho_i, \;j=u+1,\dots,i,$$
so that $g_j\geq \varrho_i-\alpha_j$ for $j=1,\ldots,i$.
It follows that the witness vector $\widetilde{g}^i=(\widetilde{g}^i_1,\ldots,\widetilde{g}^i_i)$, with $\widetilde{g}^i_j=\varrho_i-\alpha_j$ for $j=1,\ldots,i$, fits its slack:
$$|\widetilde{g}^i|=\sum_{j=1}^i\widetilde{g}^i_j\leq\sum_{j=1}^ig_j\leq p_{i+1}-1.$$

{{\em Proof of the Claim}: We prove the claim by double induction on $|U|\ge 0$ and  $i\ge 2$.

For $|U|=0$ there is nothing to prove whatever is $i\ge 2$. Let $|U|\ge 1$.
For $i=2$, one has $\nu_1-\alpha_{1}\ge 0$ and $u=2$ with $\nu_2<\alpha_2$. Since $(\nu_1-\alpha_{1}) +(\nu_2-\alpha_{2})\ge 0$ and $\nu_2=\alpha_{2}-\epsilon_2$, for some $\epsilon_2>0$, we may write
  $$(\nu_1-\alpha_{1}) +(\nu_2-\alpha_{2})=[(\nu_1-\epsilon_2)-\alpha_{1}]+(\alpha_{2}-\alpha_{2})=(\nu_1-\epsilon_2)-\alpha_{1}\ge 0.$$
  Thus $\mu_1:=\nu_1-\epsilon_2\ge \alpha_{1}\ge \mu_2:=\alpha_{2}>\nu_2\ge \nu_3$.

  Let $i=m+1\ge 3$, and $u\in\{2,\dots,m+1\}$ where $\nu_u=\alpha_{u}-\epsilon_u$, for some $\epsilon_u>0$, and $\nu_v-\alpha_v\ge 0$, $u<v\le m+1$. We distinguish two situations:

$(a)$ $u=2$: $\nu_1>\alpha_1$, $\nu_2=\alpha_2-\epsilon$, for some $\epsilon>0$, and $\nu_j\ge \alpha_j$, for $3\le j\le i$.
We have $\alpha\preceq \nu$ and we may write
$$(\nu_1-\alpha_1)+(\nu_2-\alpha_2)=(\mu_1-\alpha_{1})+(\alpha_2-\alpha_2)=(\mu_1-\alpha_{1})+(\mu_2-\alpha_2)\ge 0,$$ where $\mu_1:=\nu_1-\epsilon\ge \alpha_1, \;\mu_2:=\alpha_2$. Also $\mu_1\ge \alpha_1\ge\mu_2=\alpha_2>\nu_2\ge\nu_3\ge\cdots\ge\nu_i\ge\mu_{i+1}.$

$(b)$ $u>2$: $\nu=\alpha_u-\epsilon_u$ for some $\epsilon_u>0$ and $\nu_j\ge\alpha_j$, $u<j\le i$.
One has $\alpha\preceq \nu$, henceforth
  $$\sum_{j=1}^{u}(\nu_j-\alpha_{j})=\left[\left(\sum_{j=1}^{u-1}(\nu_j-\alpha_{j})\right)-\epsilon_u\right]+ (\alpha_{u}-\alpha_{u})\ge 0.$$
Thus $\mu_u:=\alpha_{u}>\nu_u\ge\nu_{u+1}\ge\dots\ge\nu_i\ge \nu_{i+1}$ and $\sum_{j=1}^{u-1}(\nu_j-\alpha_{j})\ge\epsilon_u>0$.

 Since $2\le u-1\le i-1\le m$, by induction, there exist $\nu'_1\ge\dots\ge\nu'_{u-1}$ with $\nu'_j\ge \alpha_{j}$, $j=1,\dots,u-1$, such that
$$\sum_{j=1}^{u-1}(\nu_j-\alpha_{j})=\sum_{j=1}^{u-1}(\nu'_j-\alpha_{j})\ge \epsilon_u.$$
 Indeed, one has  $\nu'_j=\alpha_{j}+\epsilon_j$, with $\epsilon_j\ge 0$, $j=1,\dots,u-1$, such that $\sum_{j=1}^{u-1}\epsilon_j\ge \epsilon_u$.
Define recursively the non negative integers $$\delta_j:=min(\epsilon_j,\epsilon_u-\underset{j+1\le q\le u-1}{\sum}\delta_q), \quad\text{for}\quad j=u-1,\dots,1,$$
and put $\mu_j:=\nu'_j-\delta_j=\alpha_{j}+(\epsilon_j-\delta_j)\ge 0, \quad\text{for}\quad j=u-1,\dots,1.$
Therefore, there exists $1\le u_0<u$ such that $0<\delta_{u_0}\le \epsilon_{u_0}$ and
$$\mu_j=\left\{ \begin{array}{rccl}
\alpha_{j}& & & \quad u_0< j<u \\
\alpha_{u_0}&+&(\epsilon_{u_0}-\delta_{u_0})& \\
\nu'_j & & &\quad
 1\le j<u_0.
\end{array}\right.$$
Hence, $$\mu_1\ge\dots\ge\mu_{u_0+1}\ge\nu'_{{u_0}}>\mu_0\ge \alpha_{u_0}\ge \mu_{u_0+1}= \alpha_{u_0+1}\ge \dots\ge \mu_{u-1}=\alpha_{u-1}\ge \mu_u=\alpha_{u},$$ as required.}

\bigskip

The {\em ``if"} part. Let $\varrho_i=\sum_{q=i+1}^{\ell(\alpha)}\alpha_q-p_{i+1}+1>0$, $1\leq i\leq \ell(p)-1$.
Suppose now that there is an i-witness vector $\tilde g^i=(\tilde g_1^i,\dots,\tilde g_i^i)$  of $R_\alpha^p$ for some
$1\leq i\leq \ell(p)-1$, with
$\displaystyle \tilde g_j^i:=\left[ \varrho_i-\alpha_j\right]_+,\;j=1,\dots, i$, such that $|\tilde g^i|\leq p_{i+1}-1$.
Let $\nu=(\nu_1,\ldots,\nu_{i+1},\nu_{i+2})$ be the partition of $|\alpha|$ formed by the rearrangement of the composition
 \begin{equation}\label{eq_partition}
 (\alpha_1 +\tilde g_1^i,\dots,
\alpha_i +\tilde g_i^i,\varrho_i,p_{i+1}-1- |\tilde g^i|),
\end{equation}
where $\alpha_1 +\tilde g_1^i,\dots,
\alpha_i +\tilde g_i^i\ge \nu_{i+1}=\varrho_i\ge \nu_{i+2}=p_{i+1}-1- |\tilde g^i|.$

 We will show that $\nu$ is a partition in the Schur interval of the ribbon $R_{\alpha}^p$ that is not in its support.
Indeed, the inequality $|\tilde g^i|\leq p_{i+1}-1$ shows that all entries in \eqref{eq_partition} are non negative, and $\displaystyle \sum_{q=1}^{i+2}\nu_q=\sum_{q= 1}^{\ell(\alpha)}\alpha_q=|\alpha|$. Thus, $\nu$ is well defined and is a partition of $|\alpha|$.

 Recall that $\displaystyle \varrho_i-1=\sum_{q\geq i+1}\alpha_q-p_{i+1}$ is the total number of columns of $\displaystyle R_{\alpha}^p\setminus\left(\cup_{q=1}^i<\alpha_q>\right)$ and that
 $p_{i+1}$ is the number of columns of length two in this  same ribbon. Therefore, we have $\varrho_i> p_{i+1}-1- |\tilde g^i|$. Moreover,
 for each $1\leq j\leq i$, we have
 \begin{equation}\label{eq_desig}
 \alpha_j+\tilde{g}^i_j=\begin{cases}\varrho_i,\text{ if }\varrho_i>\alpha_j\\ \alpha_j,\text{ if }\varrho_i\leq\alpha_j\end{cases}.
 \end{equation}
 It follows that $\alpha_j+\tilde{g}^i_j\geq \varrho_i$. This means that the last two entries of $\nu$ are $\nu_{i+1}=\varrho_i$ and $\nu_{i+2}=p_{i+1}-1- |\tilde g^i|$.
 In particular, it follows from Corollary \ref{char}, $(b)$,  that $\nu$ is not in the Schur support of $R_{\alpha}^p$.

 It remains to prove that $\nu$ is a partition in the Schur interval $[\alpha,(|\alpha|-p_1,p_1)]$. We start by showing that $\alpha\preceq \nu$.
From \eqref{eq_desig}, we find that for each $1\leq k\leq i$,
 $$\sum_{j=1}^k\nu_j=\sum_{j=1}^k(\alpha_j+\tilde{g}^i_j)\geq\sum_{j=1}^k\alpha_j$$
 and since $\varrho_i\geq \alpha_{i+1}$,
 $\sum_{j=1}^{i+1}\nu_j=\sum_{j=1}^{i}(\alpha_j+\tilde{g}^i_j)+\varrho_i\geq\sum_{j=1}^{i}\alpha_j+\varrho_i\geq
 \sum_{j=1}^{i+1}\alpha_j.$
 Finally, since $\nu$ is a partition of $|\alpha|$, we get  $\alpha\preceq\nu$.
 To prove that we also have $\nu\preceq (|\alpha|-p_1,p_1)$, notice that by
 \eqref{eq_desig} and Remark \ref{re:varrho}, $\nu_1$ is either equal to $\varrho_1$ or to $\alpha_1$,   and $\varrho_1\leq |\alpha|-p_1$. Therefore, we have $\nu_1\leq|\alpha|-p_1$.
 Clearly, $\nu_1+\nu_2\leq |\alpha|$, from which it follows that
$\nu\preceq (|\alpha|-p_1,p_1)$.
$\Box$

\begin{remark} \label{re:equiv1} Let $\alpha$ be a partition with parts $\ge 2$ and  overlapping partition $p$ with  $\ell(p)\ge 2$. Recall Definition \ref{witness},
$\displaystyle \varrho_i=1+\sum_{q=i+1}^{\ell(\alpha)}\alpha_q-p_{i+1}>0$, and $\tilde g^i=\{\tilde g^i_j\}_{j=1}^i=\{\left[ \varrho_i-\alpha_j\right]_+\}_{j=1}^i$, for $1\le i\le \ell(p)-1$. Observe that the following are equivalent:

$(a)$ for some $1\le i\le \ell(p)-1$,  the  size of the $i$-witness vector $\tilde g^i$  fits its slack,   that is,
\begin{equation}\label{lineq}|\tilde g^i|=\sum_{j=1}^i\left[ \varrho_i-\alpha_j\right]_+\le p_{i+1}-1.\end{equation}

$(b)$ for some $1\le i\le \ell(p)-1$, there exist  integers
$g_1,\dots,g_i\ge 0$ with $\sum_{j=1}^i g_j\le p_{i+1}-1$,  such that
\begin{equation}\label{lineq+}\alpha_j+g_j\ge 1+\sum_{q=i+1}^{\ell(\alpha)}\alpha_q-p_{i+1}\Leftrightarrow g_j\ge \varrho_i-\alpha_j,\quad j=1,\dots, i.\end{equation}

Indeed, \eqref{lineq+}  says that, for $1\le i\le \ell(p)-1$, other "witness vectors" $g=\{ g_j\}_{j=1}^i$ can be found depending on how big is the slack $p_{i+1}-1$. Simultaneously \eqref{lineq+} tells that the selected witness $\tilde g^i$ in Definition \ref{witness} is  entrywise the smallest,
$$\tilde g^i_j\le g_j,\; j=1,\dots,i, \Rightarrow |\tilde g^i|\ge |g|.$$
If our selected witness $\tilde g^i$ does not fit (is over the size of) its slack, no other (any other) choice for the witness vector will fit (oversize) it.

In the conditions of $(b)$, it can be shown that
$(\alpha_1 +g_1,\dots,
\alpha_i +g_i,\varrho_i)^+$ with $\sum_{j=1}^i g_j=$ $ p_{i+1}-1$ ($g$ has the possible  biggest size)  is a partition of $|\alpha|$ in the Schur interval  of $R_\alpha^p$ but  not in the support of $R_\alpha^p$.
\end{remark}

\begin{example} $(a)$ Consider the same example as before,  $\alpha=(7,6,6,2,2,2,2)$ and the ribbon $R_{\alpha}^p$ with $p_1=6$. Applying  Theorem \ref{Tchar} with $i=3$, one has $p_{i+1}=3$,  $\varrho_3=6$ and the $3$-witness vector  $\tilde g^3=(\tilde g^3_1, \tilde g_2^3, \tilde g_3^3)=(0,0,0)$, satisfy $|\tilde g^3|\leq p_{i+1}+1=4$. Therefore, the support $[R_{\alpha}^p]$ is not the full Schur interval. The partition
$$(7-\tilde g^3_1,6-\tilde g^3_2,6-\tilde g^3_3,\varrho_3,p_4-1-|\tilde g^3|)=(7,6,6,6,2)$$
is in the Schur interval $[\alpha,(27-6,6)]$ but not in the support $[R_{\alpha}^p]$.

$(b)$  Furthermore, considering $g_1+g_2+g_3=2=p_4-1$, with $g_i\ge 0$, $i=1,2,3$, the partitions
 $\nu_1=(6+2,7,6,2+2+2+2-2),$ $\nu_2=(6+1,7,6+1,2+2+2+2-2)$ and $\nu_3=(7+2,6,6,2+2+2+2-2)$ are in the interval $[\alpha,(27-6,6)]$ but not in the support of  $R_{\alpha}^p$.
\end{example}

\subsection{Proof of Remark \ref{aftertheor} and Corollary \ref{cor:full}}. Theorem \ref{corf} is logically equivalent  to Theorem \ref{Tchar} and says that if every  $i$-witness $\tilde g^i$ vector of $R^p_\alpha$, for $i=1,\dots,\ell(p)-1$,  is  oversized, with respect to its slack $p_{i+1}-1$, then  $R^p_\alpha$ has full Schur support.
In particular,
$R_\alpha^p$ has full support only if $\alpha_i< \varrho_i$ for every $1\le i\le \ell(p)-1$. In fact, if, for some $k\in \{1,\dots, \ell(p)-1\}$, $\alpha_k\ge \varrho_k$, then $\alpha_1\ge\cdots\ge\alpha_k\ge \varrho_k$ and
$|\tilde g^k|=\sum_{j=1}^k\left[ \varrho_k-\alpha_j\right]_+=0\le p_{k+1}-1.$
This implies that $(\alpha_1,\dots,\alpha_k,\varrho_k,p_{k+1}-1)\in[\alpha,(|\alpha|-p_1,p_1)]
$ is not in $[R_\alpha^p]$ which is absurd.

$(a)$  When  $\ell(p)=2$, one has $p=(2,1,0^{\ell(\alpha)-2})$, and $[R_\alpha^p]=[\alpha,(|\alpha|-2,2)]$ if and only if
 $\alpha_1<\varrho_1\Leftrightarrow \alpha_1<\sum_{q=2}^{\ell(\alpha)}\alpha_q$. In fact, if $\ell(p)=2$, \eqref{fullineq2}
means
$$ \left[ \varrho_1-\alpha_1\right]_+\ge 1\Leftrightarrow \varrho_1-\alpha_1>0\Leftrightarrow \varrho_1>\alpha_1\Leftrightarrow \alpha_1<1+\sum_{q=2}^{\ell(\alpha)}\alpha_q-1=\sum_{q=2}^{\ell(\alpha)}\alpha_q.$$
$(b)$ When  $\ell(p)=3$, one has $p=(3,2,1,0^{\ell(\alpha)-3})$, and $[R_\alpha^p]=[\alpha,(|\alpha|-3,3)]$ if and only if
$\alpha_1<\sum_{q=2}^{\ell(\alpha)}\alpha_q-2\;\;\text{and}\;\;
\alpha_2<\sum_{q=3}^{\ell(\alpha)}\alpha_q$. In fact, if $\ell(p)=3$, \eqref{fullineq3}
means $$  \varrho_1-\alpha_1\ge 2\Leftrightarrow \varrho_1>\alpha_1+1\Leftrightarrow 1+\sum_{q=2}^{\ell(\alpha)}\alpha_q-2>\alpha_1+1\Leftrightarrow\sum_{q=2}^{\ell(\alpha)}\alpha_q-2>\alpha_1,$$
$$  [\varrho_2-\alpha_1]_++[\varrho_2-\alpha_2]_+\ge 1
\Leftrightarrow \varrho_2-\alpha_2\ge 1 \Leftrightarrow \varrho_2>\alpha_2\Leftrightarrow \alpha_2<1+\sum_{q=3}^{\ell(\alpha)}\alpha_q-1=\sum_{q=3}^{\ell(\alpha)}\alpha_q.$$
$\Box$

\begin{example}\label{exDR}

Let $\alpha=(4,3,2,2)$  with $p=(3,2,1,0)$. We  use the characterization given by the Theorem \ref{corf} $(b)$ to prove that $R_{\alpha}^p$ has full support $[\alpha,(8,3)]$. Since $\ell(\alpha)=4$ and  $\ell(p)=3$, we have two inequalities to check:
$$\alpha_2+\alpha_3+\alpha_4-2> \alpha_1\Leftrightarrow7-2> 4,
\;\;\alpha_2<\alpha_3+\alpha_4\Leftrightarrow 3<4.$$
\end{example}


\section{Connected ribbons with full equivalence class and full Schur support}
\label{sec:fullequiv}
Building on \cite{Mn}, M. Gaetz, W. Hardt and S. Sridhar have introduced in \cite{fullequiv} the family of connected ribbons with full equivalence class.

\begin{definition} \cite[Definition 7]{fullequiv}  Let $\alpha$ be a partition with parts $\geq 2$ and $\ell(\alpha)\geq 3$. The connected ribbon $R_{\alpha}$ is said to have {\em full equivalence class} if $[R_{\alpha}]=[R_{\beta}]$, for any rearrangement $\beta$ of the entries of $\alpha$.
\end{definition}
\begin{definition}\cite{fullequiv}
Three integers $x\leq y\leq z$ are said to satisfy the {\em strict triangle inequality} if $z<x+y$. In this case, the multiset $\{x,y,z\}$ is said to satisfy the strict triangle inequality.
\end{definition}
The set of connected ribbons with full equivalence class have partitions as representatives.
 For monotone connected ribbons, the inequality \eqref{N}, in Theorem \ref{th:1.2}, \cite[Theorem II.1]{fullequiv}, giving a necessary condition for full equivalence class, is equivalent to inequality \eqref{fullineq}, in Theorem \ref{corf}, characterizing the full Schur support.

\noindent{\bf Proof of Lemma \ref{lem:equiv} } Let $ j\in\{1,\dots, \ell(\alpha)-2\}$ and $\displaystyle N_j:=max\{k:\sum_{\begin{smallmatrix}1\le i\le j\\\alpha_i<k\end{smallmatrix}}(k-\alpha_i)\le \ell(\alpha)-j-2\}.$ From the definition of $N_j$, one has  $\displaystyle\sum_{\begin{smallmatrix}1\le i\le j\\\alpha_i<N_j\end{smallmatrix}}(N_j-\alpha_i)\le \ell(\alpha)-j-2$. Then
$\displaystyle N_j<\varrho_j\Leftrightarrow\sum_{\begin{smallmatrix}1\le i\le j\\\alpha_i<\varrho_j\end{smallmatrix}} (\varrho_j-\alpha_i)\ge \ell(\alpha)-j-1.$
$\Box$

 \noindent{\bf Proof of Theorem \ref{fullequivschur}.}
 Because $R_\alpha$ is connected, $p=(\ell(\alpha)-1,\dots,2,1,0)$ and, in Definition \ref{witness}, $\varrho_j=\sum_{q=j+1}^{\ell(\alpha)}\alpha_q-(\ell(\alpha)-j-2)$, for $j\in\{1,\dots,\ell(\alpha)-2\}$. Suppose that $R_\alpha$ does not have full Schur support. Then  Theorem \ref{corf} says that for some $t\in\{1,\dots,\ell(\alpha)-2\} $,
 \begin{equation}\label{equivschur}\sum_{\begin{smallmatrix}1\le i\le t\end{smallmatrix}}[\varrho_t-\alpha_i]_+=\sum_{\begin{smallmatrix}1\le i\le t\\\alpha_i<\varrho_t\end{smallmatrix}}(\varrho_t-\alpha_i)\le \ell(\alpha)-t-2.\end{equation}
 Inequality \eqref{equivschur} implies in the definition of $N_t$, \eqref{N}, that $N_t\ge \varrho_t$  with $t\in\{1,\dots,\ell(\alpha)-2\} $. Henceforth, by Theorem  \ref{th:1.2}, \cite[Theorem 1.2]{fullequiv}, one concludes that $\alpha$ does not have  full equivalence class. When $\ell(\alpha)=3$, by Theorem \ref{Tchar}, $R_\alpha$ has full support $[\alpha,(|\alpha|-2,2)]$ if and only if $\alpha_1<\alpha_2+\alpha_3$ (strict triangle inequality). Theorem 3.4, in \cite{fullequiv2}, also shows that $R_\alpha$ has full equivalence class if and only if $\alpha_1<\alpha_2+\alpha_3$. When $\ell(\alpha)=4$, by Theorem \ref{Tchar} $(b)$, $R_\alpha$ has full support $[\alpha,(|\alpha|-3,3)]$ if and only if  \eqref{fullineq3} are satisfied. Theorem 3.6, in \cite{fullequiv2}, also shows that $R_\alpha$ has full equivalence class if and only if \eqref{fullineq3} are satisfied.
 $\Box$

 Next  theorem gives  a sufficient condition for a monotone connected ribbon to have full equivalence class \cite[Corollary1.4]{fullequiv} which in turn, thanks to  Theorem \ref{fullequivschur},   also gives  a sufficient condition for monotone connected ribbons to have  full Schur support.

\begin{theorem}\label{triangle}
Let $\beta=(\beta_1,\ldots,\beta_{\ell(\beta)})$ be a composition with parts $\geq 2$  and $\ell(\beta)\geq 3$. If all
 3-multisets contained in  $\{\beta_1,\dots,\beta_{\ell(\beta)}\}$ satisfy the strict triangle inequality then the connected ribbon $R_{\beta}$ has

$(a)$ \cite[Corollary 1.4]{fullequiv}  full equivalence class; and

$(b)$  full Schur support $[\beta^+,(|\beta|-\ell(\beta)+1,\ell(\beta)-1)]$.
\end{theorem}

The strict triangle inequality condition given by the previous theorem is sufficient for a connected ribbon to have full support, but it is not necessary. For instance, not all 3-subsets of the partition $\alpha=(4,3,2,2)$  satisfy the strict triangular inequality $(4=2+2)$, but as we have seen in Example \ref{exDR}, the connected ribbon $R_{\alpha}^p$ has full support.
Nevertheless, for partitions $\alpha$ with length $3$ the connected ribbon $R_{\alpha}^p$ has full support (full support) if and only if $\alpha$ satisfy the strict triangular inequality \eqref{fullineq2}.

 Next statement classifies  arbitrary compositions with length $3$ with respect to the full support where we may verify that for non monotone compositions the strict triangular inequality is not a necessary condition. This means that the full Schur support classification  for non monotone compositions  and  for partitions is not the same.

\begin{corollary} \label{cor:length3}
Let $\beta$ be a composition of length 3 with each part $\geq 2$. Then, the connected ribbon $R_{\beta}$ has full support except when $\beta=(\beta_1^+,\beta_2,\beta_3)$ or $\beta=(\beta_2,\beta_3,\beta_1^+)$ with $\beta_1^+\ge \beta_2+\beta_3$, in which cases, the partition $\nu=(\beta_1^+, \beta_2+\beta_3)$ is in the Schur interval  but  not in the support of $R_{\beta}$.
\end{corollary}
\begin{proof}
By the previous theorem, we know that if $\beta$ satisfies the strict triangle inequality,  $\beta_1^+< \beta_2+\beta_3$, then $R_{\beta}$ has full support.
There remains three cases to analyse: $\beta=(\beta_1^+,\beta_2,\beta_3)$, or $\beta=(\beta_2,\beta_3,\beta_1^+)$, or $\beta=(\beta_2,\beta_1^+,\beta_3)$, with $\beta_1^+\ge \beta_2+\beta_3$.
Since the support of $R_{\beta}$ is invariant under 180 degrees rotation of the ribbon $R_{\beta}$, the first two cases can be reduced to the first one.
Suppose that  $\beta=(\beta_1^+,\beta_2,\beta_3)$ satisfies $\beta_1\geq\beta_2+\beta_3$ and $\beta_2\geq\beta_3$, and  recall that the overlapping partition is $p=(2,1,0)$.   Applying Theorem \ref{Tchar} with $i=1$ one has $\varrho_1=\beta_2+\beta_3$ and  $\tilde g^1_1=0\leq p_2-1=0$,
 and henceforth, the support of $R_{\beta}^p$ is not the entire Schur interval, since the partition $\nu=(\beta_1, \beta_2+\beta_3)$ is in the Schur interval  but  not in the support of $R_{\beta}$.
The same partition $\nu$ proves the result when $\beta_3\geq\beta_2$. Note that an LR filling of $R_{(\beta_1^+,\beta_2,\beta_3)}$ with content $\nu$ would oblige to fill the first row with $\beta_1$ 1's and  the last two rows with    $\beta_2+\beta_3$ 2's. Since the two last rows of $R_\beta$ overlap such a filling violates the column semistandard condition.

Finally, in the case of the connected ribbon $R_{(\beta_2,\beta_1^+,\beta_3)}$ satisfying $\beta_1^+\ge \beta_2+\beta_3$,  it is easy to show that    $\rm LR_{R_{\beta},\nu}\neq \emptyset$  for any partition in the Schur interval $[\beta^+,(|\beta|-2,2)]$. Indeed if $\ell(\nu)=3$, any $T\in Tab(\nu,\beta)$ is such that $\mathcal{D}(\widehat T)=\mathcal{S}(\beta)=\{\beta_2,\beta_1^++\beta_2\}$. If $\ell(\nu)=2$, consider the canonical filling
 $T\in Tab(\nu,\beta)$, $\nu=(\nu_1,\nu_2)$. Then the second row of $T$ has $3$'s and because $\beta_1^+\ge \beta_2+\beta_3$, $\nu_2< \beta_1^++\beta_3$ (otherwise $\nu_2>\nu_1=\beta_2$), the first row of $T$ has $\beta_2$ $1$'s and  at least one two. In case, $\nu_2=\beta_3\ge 2$, the 2nd row of $T$ has $\beta_3$ $3$'s and the $\beta_1^+\ge 2$, $2$'s are all in the first row of $T$, in which case we swap the rightmost $2$ in the first row with the leftmost $3$ in the second row to get a new tableau in $Tab(\nu,\beta)$. The descent set of this new tableau
 is $\mathcal{S}(\beta)=\{\beta_2,\beta_1+\beta_2\}$.
\end{proof}
\begin{remark}\label{compfull} If the composition  $\beta=(\beta_2,\beta_1^+,\beta_3)$ satisfies $\beta_1^+\ge \beta_2+\beta_3$  the connected ribbon $R_\beta$ has full support while $R_{\beta^+}$ does not have full support because $\beta_1^+\ge \beta_2+\beta_3$.
\end{remark}

\begin{corollary} \cite[Theorem 1.5.]{PvW} Let $\beta$ be an arbitrary composition with parts $\ge 1$.
Connected ribbons $R_\beta$ whose  column and row lengths differ at most  one have full support. They also have full equivalence class except when $\beta=(2^{\ell(\beta)-1},1)$, $\ell(\beta)\ge 3$.
\end{corollary}

\begin{proof} Let $\beta=(\beta_1,\dots,\beta_{\ell(\beta)})$ and $R_\beta$  a connected ribbon in the conditions of the statement. Observe that the transpose of $R_\beta$ is still in the conditions of the statement. If $R_\beta$ or its transpose consists only of one or two rows  is trivial. Suppose that $R_\beta$ has at least three rows.
   If $\beta_i\geq 2$ for all $1\leq i\leq \ell(\beta)$, then  $|\beta_i-\beta_j|\leq 1$, for all
$1\leq i,j\leq \ell(\beta)$, and  any three parts $\beta_i\leq\beta_j\leq\beta_k$ of $\beta$ satisfy the strict triangle inequality
$\beta_k<\beta_i+\beta_j.$
By   Theorem \ref{triangle}, $(b)$,  $R_{\beta}$ has full support and full equivalence class. If $\beta_1=\beta_{\ell(\beta)}=1$ then $\beta_i=2$, $1<i<\ell(\beta)$, and transposing $R_\beta$ we fall in one of the previous cases: $\beta=(2^{\ell(\beta)})$ with $\ell(\beta)\ge 2$, and again $R_\beta$ has full support and full equivalence class.
If $\beta_1=1<\beta_{\ell(\beta)}$ or $\beta_1>\beta_{\ell(\beta)}=1$, by 180 degrees-rotation, we may assume the last inequality and we have $\beta=(2^{\ell(\beta)-1},1)$ with $\ell(\beta)\ge 3$. Put $s:=\ell(\beta)-1$ and  let $I:=[(2^s),(s,s)]$ be the Schur interval of $R_{(2^s)}$, $s\ge 2$. By the previous cases, the support of $R_{(2^s)}$ is the full interval $I$.

 The Schur interval of $R_{(2^s,1)}$ is $[(2^s,1),(s+1,s)]$ and it is self conjugate. Its  partitions are obtained using one extra box in the construction of the elements of $I$. There are three possible positions  to put the extra box in one element of $I$ and obtain $\nu\in [(2^s,1),(s+1,s)]$:
$(a)$ far right of the first row; $(b)$ below the last row;
or $(c)$  far right of the last row.

 Because $R_\beta=R_{(2^s,1)}={(R_\beta})'$ and $c_{R_\beta}^\nu=c_{R_\beta}^{\nu'}$,  by transposition of $\nu$, we may reduce $(a)$ to $(b)$. Hence if $T\in LR_{R_{(2^s)},\mu}$ then the SSYT $T_\Box$, obtained  by adding one box filled with $s+1$ below the last row of $T$, is in $LR_{R_{(2^s,1)},\nu}$ with $\nu=(\mu,1)$. Note that $\mathcal{D}(\widehat T_\Box)=\mathcal{S}(2^s,1)=\mathcal{S}(2^s)\cup\{2s\}$. It remains to prove that $\nu=(\mu_1,\dots,\mu_{\ell(\mu)-1},\mu_{\ell(\mu)}+1)$ obtained in $(c)$ is in $[R_{(2^s,1)}]$. If the last row of $T$ has at most one $s$ then just add one box filled with $s+1$ at the end of the this row to obtain $T_\Box$. If the last row of $T$ has two $s$'s also add one box filled with $s+1$ at the end of  this row. 
 At least one entry in the row above is not in the last row and choose that in the rightmost position: it is the far right entry:

 $(i)$ $s-1$,
 $T=\begin{array}{ccccc}\cdots\\\cdots&a&b&(s-1)\\
 \cdots&s&s&(s+1)\\
 \end{array}\rightarrow T_\Box=\begin{array}{ccccc}
 \cdots\\\cdots&a&b&s\\
 \cdots&s-1&s&(s+1)\\
 \end{array}, a<s-1, b\le s-1$

$(ii)$ $a<s-1$, $T=\begin{array}{ccccc}\cdots\\\cdots&d&c&a\\
 \cdots&s&s&(s+1)\\
 \end{array}\rightarrow T_\Box=\begin{array}{ccccc}
 \cdots\\
 \cdots \cdots&d&c&s\\
 \cdots a&\cdots&s&(s+1)\\
 \end{array}, c\le a<s-1, d<a$,
 $a$ enters in the last row of $T$ bumping to the right the left most strictly bigger entry;
 otherwise,
 $T=\begin{array}{ccccc}\cdots\\\cdots&d&c&a&x\\
 \cdots&x&s&s&(s+1)\\
 \end{array}\rightarrow T_\Box=\begin{array}{ccccc}
 \cdots\\
 \cdots &d&c&x&s\\
 \cdots a&\cdots&x&s&(s+1)\\
 \end{array}, d< a<x\le s-1, c\le a$, $a$ enters in the last row of $T$ bumping to the right the left most strictly bigger entry. In any case and $\mathcal{D}(\widehat T_\Box)=\mathcal{S}(2^s,1)$.

 Indeed, $\beta=(2^{\ell(\beta)-1},1)$ and $\gamma=(2,1,2^{\ell(\beta)-2})$, $\ell(\beta)\ge 3$, do not have the same Schur interval. The Schur interval of the latter is $[(2^s,1); (3,2^{s-2},1^2)]$ with $s=\ell(\beta)-1$ and henceforth $\beta=(2^{\ell(\beta)-1},1)$, $\ell(\beta)\ge 3$ does not have full equivalence class.
\end{proof}

\section{ Towards to  a coincidence between full  Schur support monotone connected ribbons and  full equivalence classes}
\label{conj}
In this section we consider connected ribbons with parts $\ge 2$ arranged in any order.
The necessary condition, given by Theorem \ref{ThmNec}, for the LR coefficient $c_{R_{\alpha}}^{\nu}$ to be  positive, with $\alpha$  a partition,  is generalized  to a  connected ribbon $R_{\alpha_\pi}$ where $\alpha_\pi$, $\pi\in \Sigma_{\ell(\alpha)}$, is a $\pi$-permutation   of the entries of $\alpha$. Thanks to the $180^\circ$-rotation symmetry of LR coefficients, $c_{R_\alpha}^\nu=c_{(R_\alpha)^\circ}^\nu$,  it is sufficient to consider partitions $\alpha$ of length $\ge 3$. That is, we already know that $c_{R_{(\alpha_1,\alpha_2)}}^\nu=c_{R_{(\alpha_2,\alpha_1)}}^\nu>0\Leftrightarrow\nu_i\leq \sum_{q= i}^{2}\alpha_q-p_i,\;\;1\leq i\leq 2$.
Recall  the definition of overlapping partition of a connected ribbon with row lengths in arbitrary order, Definition \ref{overlap}, and that the overlapping partition $p^\pi=(p_1^\pi,p_2^\pi,\dots,p_{\ell(\alpha)}^\pi,0)$ of the connected ribbon $R_{\alpha_\pi}$ satisfies \eqref{overlap}, $p^\pi\subseteq (\ell(\alpha)-1,\dots,1,0)$, that is, $p_1^\pi=\ell(\alpha)-1$, and $p_i^\pi\le \ell(\alpha)-i$, $2\le i\le \ell(\alpha)$.
\begin{theorem}\label{ineqcomposition} Let $\alpha$ be a partition with parts $\ge 2$, and $R_{\alpha_\pi}$ a connected ribbon with overlapping partition $p^\pi$.
  Let $\nu\in[\alpha,(|\alpha|-\ell(\alpha)+1,\ell(\alpha)-1)]$. Then
 \begin{equation}\label{equgen}\nu\in[R_{\alpha_\pi}](c_{R_{\alpha_\pi}}^\nu>0)\Rightarrow\nu_i\leq \sum_{q= i}^{\ell(\alpha)}\alpha_q-p^\pi_i,\;\;1\leq i\leq \ell(p^\pi).\end{equation}

\end{theorem}
\begin{proof} We prove the contrapositive assertion: if there exists $i\in\{1,\dots,\ell(\alpha)-2\}$ such that $\nu_{i+1}\ge \sum_{q\ge i+1}^{\ell(\alpha)}\alpha_q-p_{i+1}^\pi+1$ then $c_{R_{\alpha_\pi}}^\nu=0$. (Indeed $\nu_{1}\le \sum_{q\ge 1}^{\ell(\alpha)}\alpha_q-p_{1}^\pi+1$ and $\nu_{\ell(\alpha)}\le \alpha_{\ell(\alpha)}$.)

Let $\alpha_\pi=(\beta_1,\dots,\beta_{\ell(\alpha)})$ and let $i$ be the smallest element in $\{1,\dots,\ell(\alpha)-2\}$ such that $\nu_{i+1}\ge \sum_{q\ge i+1}^{\ell(\alpha)}\alpha_q-p_{i+1}^\pi+1$. Since $|\nu|=|\alpha|$ and $\alpha\preceq \nu$, one has
\begin{equation}\label{eq}
\sum_{q=1}^{i}\beta_q\le \sum_{q=1}^{i}\alpha_q\le \sum_{q=1}^{i}\nu_q=|\alpha|-\sum_{q\ge i+1}^{\ell(\alpha)}\nu_q\le \sum_{q=1}^{i}\alpha_q+p_{i+1}^\pi-1.
\end{equation}
If we place $\nu_1$ $1$'s, $\nu_2$ $2$'s, $\dots, \nu_i$ $i$'s in $R_\beta$ to obtain an LR filling then at least the first $i$ rows of $R_\beta$ are completely filled because one can not place in them numbers $\ge i+1$. Henceforth, in the best case one has $\sum_{q=1}^{i}\beta_q= \sum_{q=1}^{i}\alpha_q= \sum_{q=1}^{i}\nu_q$, so that  it remains $\ell (\alpha)-i$ rows of $R_\beta$ to place $\nu_{i+1}$ $i+1$'s. Because $R_\beta$ is connected the number of columns of length two among them is
$ \ell(\alpha)-i-1\ge p_{i+1}^\pi$. (In fact, in this case, one has the equality $ \ell(\alpha)-i-1=p_{i+1}^\pi$. Because one has the equality of  the multisets $\{\beta_i,\dots,\beta_{\ell(\alpha)}\}=\{\alpha_i,\dots,\alpha_{\ell(\alpha)}\}$  and by definition $p_{i+1}^\pi$ is the number of columns of length two among the rows $\alpha_i,\dots,\alpha_{\ell(\alpha)}$ of the ribbon $R_\beta$ which in this the same as among the rows of $R_{(\beta_i,\dots,\beta_{\ell(\alpha)})}$.) It means that in the best case the number of available boxes to fill with $\nu_{i+1}$, $i+1$'s, is in fact
\begin{eqnarray}\sum_{q\ge i+1}^{\ell(\alpha)}\beta_q- (\ell(\alpha)-i-1)=|\beta|-\sum_{q=1}^{i}\beta_q-(\ell(\alpha)-i-1)=|\alpha|- \sum_{q=1}^{i}\alpha_q-(\ell(\alpha)-i-1)\nonumber\\
=\sum_{q\ge i+1}^{\ell(\alpha)}\alpha_q- (\ell(\alpha)-i-1)\le \sum_{q\ge i+1}^{\ell(\alpha)}\alpha_q-p_{i+1}^\pi<\sum_{q=1}^{i}\alpha_q+p_{i+1}^\pi-1\le \nu_{i+1},\nonumber\end{eqnarray}
which is not enough. Therefore $c^\nu_{R_{\alpha_\pi}}=0$.
\end{proof}

\begin{remark}\label{re:conj}
 $(1)$ Under the assumption that row lengths are $\ge 2$,  $R_{\alpha_\pi}$ and $R_\alpha$ have the same  the Schur interval, $[\alpha,(|\alpha|-\ell(\alpha)+1,\ell(\alpha)-1)]$, for all $\pi\in \Sigma_{\ell(\alpha)}$.

 $(2)$ Assuming in Theorem \ref{ineqcomposition} that inequalities \eqref{equgen} are also sufficient for $\nu\in[R_{\alpha_\pi}]$, we have the following result.
 If $\nu\in[R_{\alpha}]$ with $\alpha$ a partition, and $\pi\in \Sigma_{\ell(\alpha)}$ then $$\nu_i\leq \sum_{q= i}^{\ell(\alpha)}\alpha_q-(\ell(\alpha)-i)\leq \sum_{q= i}^{\ell(\alpha)}\alpha_q-p_i^\pi,\;\;1\leq i\leq \ell(\nu)\Rightarrow \nu \in[R_{\alpha_\pi}].$$
Therefore, $[R_\alpha]\subseteq[R_{\alpha_\pi}]$, for any $\pi\in \sum_{\ell(\alpha)}$.
If $R_\alpha$ has full Schur support, $[R_{\alpha_\pi}]= [R_{\alpha}]$, for any $\pi\in \sum_{\ell(\alpha)}$, and $R_\alpha$ has full equivalence class. Thereby, $R_\alpha$ does not have full equivalence class if and only if $[R_\alpha]\subsetneqq[R_{\alpha_\pi}]$, for some $\pi\in \sum_{\ell(\alpha)}$.
\end{remark}
In other words, the  connected  ribbon $R_\alpha$ with $\alpha$ a partition with parts $\ge 2$ has full support only if $\alpha$ has full equivalence class.
This implies that the Gaetz-Hardt-Sridhar conjecture \cite[Conjecture II.4]{fullequiv} claiming that the necessary condition on full equivalence classes \eqref{N} is also sufficient, is true.
\begin{conjecture}Let $\alpha$ be a partition with parts $\ge 2$ and $R_\alpha $ a connected ribbon. Then the following are equivalent

$(a)$ $R_\alpha$ has full Schur support, that is, $[R_\alpha]=[\alpha,(|\alpha|-\ell(\alpha)+1,\ell(\alpha)-1]$;

$(b)$ $\alpha$ has full equivalence class;

$(c)$ For all $ j\in\{1,\dots, \ell(\alpha)-2\}$,$$ N_j:=max\{k:\sum_{\begin{smallmatrix}1\le i\le j\\\alpha_i<k\end{smallmatrix}}(k-\alpha_i)\le \ell(\alpha)-j-2\}<\varrho_j\Leftrightarrow \sum_{\begin{smallmatrix}1\le i\le j\\\alpha_i<\varrho_j\end{smallmatrix}} (\varrho_j-\alpha_i)\ge \ell(\alpha)-j-1.$$
\end{conjecture}


\end{document}